\documentclass[12pt,a4paper,reqno]{amsart}
\usepackage[a4paper,top=3cm,bottom=3cm,inner=3cm,outer=3cm,includehead,includefoot]{geometry}
\usepackage{amsmath,amssymb,amsthm,calc,verbatim,enumitem,tikz,url,hyperref,mathrsfs,cite}
\hypersetup{colorlinks=true,citecolor=blue,linktoc=page}

\newtheorem{theorem}{Theorem}
\newtheorem*{AL}{The Aizenman--Lebowitz Lemma}
\newtheorem{lemma}[theorem]{Lemma}
\newtheorem{proposition}[theorem]{Proposition}

\newtheorem{obs}[theorem]{Observation}
\theoremstyle{definition}
\newtheorem{defn}[theorem]{Definition}
\newtheorem{prob}[theorem]{Problem}

\numberwithin{theorem}{section}
\setlist[itemize]{leftmargin=1cm}
\setlist[enumerate]{leftmargin=1.2cm}

\everydisplay{\thickmuskip=7mu}
\renewcommand{\leq}{\leqslant}
\renewcommand{\geq}{\geqslant}
\renewcommand{\le}{\leqslant}
\renewcommand{\ge}{\geqslant}
\renewcommand{\to}{\rightarrow}

\let\epsilon\varepsilon
\let\eps\varepsilon

\def\eps{\varepsilon}
\def\le{\leqslant}
\def\ge{\geqslant}
\def\<{\langle}
\def\>{\rangle}

\newcommand{\ds}{\displaystyle}

\def\E{\mathbb{E}}
\def\Ex{\mathbb{E}}
\def\N{\mathbb{N}}
\def\P{\mathbb{P}}
\def\Pr{\mathbb{P}}

\def\H{\mathcal{H}}

\def\LL{\mathcal{L}}
\def\U{\mathcal{U}}

\def\Z{\mathbb{Z}}

\def\Exp{\mathrm{Exp}}

\def\Var{\mathrm{Var}}

\def\0{\mathbf{0}}
\def\k{\mathbf{k}}

\def\sh{\textup{short}}

\title[Nucleation and growth in two dimensions]{Nucleation and growth in two dimensions}

\author[B. Bollob\'as \and S. Griffiths \and R. Morris \and L.T. Rolla \and P.J. Smith]{B\'ela Bollob\'as \and Simon Griffiths \and Robert Morris \and \\ Leonardo Rolla \and Paul Smith}
\address{
B\'ela Bollob\'as \hfill\break
Department of Pure Mathematics and Mathematical Statistics, Wilberforce Road, Cambridge, CB3 0WA, UK, and Department of Mathematical Sciences, University of Memphis, Memphis, TN 38152, USA, and London Institute for Mathematical Sciences, 35a South Street, London, W1K 2XF, UK}
\email{b.bollobas@dpmms.cam.ac.uk}
\address{
    Simon Griffiths \hfill\break
    Departamento de Matem\'atica, PUC-Rio, Rua Marqu\^{e}s de S\~{a}o Vicente 225, G\'avea, 22451-900 Rio de Janeiro, Brasil
} 
\email{simon@mat.puc-rio.br}
  \address{
    Robert Morris\hfill\break
   IMPA, Estrada Dona Castorina 110, Jardim Bot\^anico, Rio de Janeiro, RJ, Brasil
 }
  \email{rob@impa.br}
 \address{
    Leonardo Rolla \hfill\break
    Departamento de Matem\'atica, Universidad de Buenos Aires, Ciudad Universitaria, 
    Capital Federal, 
    Argentina
 } 
 \email{leorolla@dm.uba.ar}
 \address{
 Paul Smith \hfill\break
 IMPA, Estrada Dona Castorina 110, Jardim Bot\^anico, Rio de Janeiro, RJ, Brasil}
\email{psmith@impa.br}

\thanks{
    This research was partially supported by: (BB) NSF grant DMS~1301614 and MULTIPLEX grant no.~317532; (SG) EPSRC grant EP/J019496/1, CNPq (Proc. 500016/2010-2 and Proc. 310656/2016-8), and a PUC-Rio Bolsa de incentivo \`a produtividade em ensino e pesquisa; (RM) CNPq (Proc.~479032/2012-2 and Proc.~303275/2013-8), FAPERJ (Proc.~201.598/2014), and ERC Starting Grant 680275 MALIG}

\begin{document}

\begin{abstract}
We consider a dynamical process on a graph $G$, in which vertices are infected (randomly) at a rate which depends on the number of their neighbours that are already infected. This model includes bootstrap percolation and first-passage percolation as its extreme points. We give a precise description of the evolution of this process on the graph $\Z^2$, significantly sharpening results of Dehghanpour and Schonmann. In particular, we determine the typical infection time up to a constant factor for almost all natural values of the parameters, and in a large range we obtain a stronger, sharp threshold. 
\end{abstract}

\maketitle

\thispagestyle{empty}
\enlargethispage{\baselineskip}

\section{Introduction}
Models of random growth on $\Z^d$ have been studied for many years, motivated by numerous applications, such as cell growth~\cite{Eden,Turing}, crystal formation~\cite{KS95} and thermodynamic ferromagnetic systems~\cite{CM13b,DS97b}. Two particularly well-studied examples of such models are bootstrap percolation (see, e.g.,~\cite{AL,BBDM,Hol,M17}) and first-passage percolation (see, e.g.,~\cite{ADH,BKS}); in the former, vertices become infected once they have at least $r$ already-infected neighbours, whereas in the latter vertices are infected at rate $1$ by each of their neighbours. 

In this paper we will study a particular family of growth processes on $\Z^2$ which interpolates between bootstrap percolation at one extreme, and first passage percolation at the other. Given parameters $n \ge 1$ and $1 \le k \le n$, each vertex $v \in \Z^2$ becomes infected randomly at rate $1/n$ if it has no infected neighbours, at rate $k/n$ if it has one infected neighbour, and at rate $1$ if it has two or more infected neighbours. (At the start all vertices are healthy, and infected vertices remain infected forever.) This model was first studied by Dehghanpour and Schonmann~\cite{DS97a}, who proved that the (random) time $\tau$ at which the origin is infected satisfies
\begin{equation}\label{eq:DSbound}
\tau \, = \, \left\{
\begin{aligned}
& \; \left( \ds\frac{n}{k} \right)^{1 + o(1)} & \textup{ if } k \le \sqrt{n}\\
& \left( \ds\frac{n^2}{k} \right)^{1/3 + o(1)} & \textup{ if } k \ge \sqrt{n}
\end{aligned} \right. 
\end{equation}
with high probability as $n \to \infty$. (We will informally refer to $\tau$ as the `relaxation time'.) They also proved similar bounds in higher dimensions, and in~\cite{DS97b} applied their techniques to study the metastable behavior of the kinetic Ising model with a small magnetic field and vanishing temperature. Corresponding results were obtained for models with $r$ different infection rates (depending on the number of infected neighbours, see Section~\ref{sec:higher:dims}) by Cerf and Manzo~\cite{CM13a,CM13b}, who adapted techniques from the study of bootstrap percolation, in particular those of~\cite{CC,CM}.

In this paper we will study the model of Dehghanpour and Schonmann in greater detail. In particular, we will determine $\tau$ up to a constant factor for almost all natural functions $k = k(n)$, and in a large range we will moreover determine a sharp threshold for $\tau$. Our main theorem is as follows. 
 
\begin{theorem}\label{thm:NG2d}
The following bounds hold with high probability as $n \to \infty$:
\begin{enumerate}
\item[$(a)$] If $k \ll \log n$ then
\[
\tau = \left(\frac{\pi^2}{18}+o(1)\right)\frac{n}{\log n}.
\]
\item[$(b)$] If $\log n \ll k \ll \sqrt{n}(\log n)^2$ then
\[
\tau = \left( \frac{1}{4} + o(1) \right) \frac{n}{k} \log \bigg( \frac{k}{\log n} \bigg).
\]
\item[$(c)$] If $\sqrt{n}(\log n)^2 \ll k \ll n$ then
\[
\tau = \Theta\left(\frac{n^2}{k \log (n/k)}\right)^{1/3}.
\]
\end{enumerate}
\end{theorem}

The bounds above are obtained via a structural description of the nucleation and growth process. In regime~$(a)$, the behaviour of the process is similar to that of 2-neighbour bootstrap percolation (see Section~\ref{sec:boot:intro}), since most nucleations do not have time to grow at all by time $\tau$, and the threshold is the same as that determined by Holroyd~\cite{Hol} in the case $k = 1$. In regime~$(b)$ the behaviour changes significantly, and individual nucleations grow (at an accelerating rate) to size almost $\sqrt{k}$ before combining together in a final (and extremely rapid) `bootstrap-like' phase. In regime~$(c)$ the behaviour changes again, as individual droplets reach their `terminal velocity' (see Section~\ref{sec:growth}) before meeting one another. 

We will give a more detailed description of the behaviour of the process in Section~\ref{sec:outline}, together with an outline of the proof. In Sections~\ref{sec:AR} and~\ref{sec:TV} we prove various bounds on the growth of a single droplet, and in Section~\ref{upper:sec} we deduce the (relatively straightforward) upper bounds in Theorem~\ref{thm:NG2d}. In Sections~\ref{sec:lower:boot},~\ref{sec:lowerAC} and~\ref{sec:lowerTV} we prove our lower bounds on $\tau$, and in Section~\ref{sec:problems} we discuss open problems, including more general update rules and higher dimensions.

\section{Outline of the proof}\label{sec:outline}

In this section we will describe the qualitative behaviour of the nucleation and growth process that lies behind the quantitative bounds in Theorem~\ref{thm:NG2d}. As a warm-up, let us consider the 1-dimensional case, in which vertices of $\Z$ are infected at rate $1$ or $1/n$ depending on whether they have an infected neighbour or not. After time $\Theta(\sqrt{n})$, a nucleation will appear at distance $d = \Theta(\sqrt{n})$ from the origin; this infection will then spread sideways at rate $1$, reaching the origin after time $d \pm O(\sqrt{d})$. It follows easily that $\tau = \Theta(\sqrt{n})$ with high probability,\footnote{More precisely, $\ds\lim_{C \to \infty} \lim_{n \to \infty} \Pr\big( \tau \ge C \sqrt{n} \big) = \ds\lim_{c \to 0} \lim_{n \to \infty}  \Pr\big( \tau \le c\sqrt{n} \big) = 0$.} and moreover that there is no sharp threshold for $\tau$. 

In two dimensions the situation is substantially more complicated, so let us begin by considering the growth of a single nucleation, which produces (after some time) a (roughly) square `droplet' growing around it. The growth of the droplet is controlled by iterated 1-dimensional processes on its sides; initially the growth is slow, since it takes (expected) time $n/km$ for a new infection to appear on a given side of an $m \times m$ square, but it accelerates as the droplet grows. When $m = \Theta(\sqrt{n/k})$ the droplet stops accelerating and reaches its `terminal velocity', at which it takes time $\Theta(\sqrt{n/k})$ to grow one step in each direction, cf. the 1-dimensional setting.   

The main component missing from the above heuristic is the interaction between different droplets. This is most important (and obvious) when $k \ll \log n$, when the process behaves almost exactly like bootstrap percolation (see Section~\ref{sec:boot:intro}), and nucleations are likely to be swallowed by a large `critical droplet' before they have time to grow on their own. Above $k = \Theta(\log n)$ this behaviour changes, and the rate of growth of each individual nucleation becomes important; however, the process still ends with a significant `bootstrap phase', in which various small droplets combine to quickly infect the origin. This bootstrap phase only reduces $\tau$ by (at most) a small constant factor in regime~$(b)$; in regime~$(c)$, on the other hand, it saves a factor of $(\log (n/k))^{1/3}$. It is only once $k = \Omega(n)$ that the first passage percolation properties of the process predominate, and our techniques break down.

We will next discuss in more detail the two key dynamics discussed above: the growth of a droplet from a single nucleation, and bootstrap percolation.

\subsection{Growth from a single nucleation}\label{sec:growth}

Consider the variant of our nucleation and growth model in which the only nucleation occurs at time zero, at the origin, and a droplet grows according to the usual $1$- and $2$-neighbour rates (i.e., $k/n$ and $1$). This model was introduced and studied by Kesten and Schonmann~\cite{KS95}, who determined the limiting behaviour of the infected droplet. In order to prove Theorem~\ref{thm:NG2d} we will require some rather more detailed information about the growth, both during the `accelerating phase' and the `terminal velocity phase'. 

Let $S(m)\subset\Z^2$ denote the square centred on the origin with $2m+1$ sites on each side, and let $T^-(m)$ and $T^+(m)$ (respectively) denote the time at which the first site of $\Z^2 \setminus S(m-1)$ is infected, and the first time at which all of $S(m)$ is infected, in the Kesten--Schonmann process. The following theorem  describes the growth of a droplet in the accelerating phase.

\begin{theorem}\label{thm:acc} 
Fix $\delta > 0$ and let $n \in \N$. For every $1 \le k = k(n) \ll n$ and every $1 \ll m = m(n) \ll \sqrt{(n / k) \log (n/k)}$, we have
\[
\left(\frac{1}{2} - \delta\right) \frac{n}{k} \log{m} \, \le \, T^{-}(m) \, \le \, T^{+}(m) \, \le \, \left(\frac{1}{2} + \delta\right) \frac{n}{k} \log{m}
\]
with high probability as $n \to \infty$. 
\end{theorem}

Theorem~\ref{thm:acc} will follow immediately from some stronger and more technical results, which we will state and prove in Section~\ref{sec:AR}, and which we will need in order to prove Theorem~\ref{thm:NG2d}. The long-term growth of the Kesten--Schonmann droplet is described by the following theorem, a proof of which can be found in~\cite{Simon}. (We refer the interested reader also to the papers~\cite{RSSV} and~\cite{SSV}, where similar results were proved in a closely related setting.) We will not need such strong bounds, however, and we will prove the results we need from first principles, see Section~\ref{sec:TV}. 

\begin{theorem}\label{thm:tv} 
Fix $\delta > 0$ and let $n \in \N$. For every $1 \le k = k(n) \ll n$ and every $m = m(n) \gg \sqrt{ n/k } \log(n/k)$, we have
\[
\bigg( \frac{1}{\sqrt{2}} -  \delta \bigg) m \sqrt{ \frac{n}{k} } \, \le\, T^{-}(m) \, \le \, T^{+}(m)\,\le \, \bigg( \frac{1}{\sqrt{2}} + \delta \bigg) m \sqrt{ \frac{n}{k} }
\]
with high probability as $n \to \infty$. 
\end{theorem}

In order to motivate the statements above, let us do a quick (and imprecise) calculation. Suppose for simplicity that our droplet is currently a rectangle with semi-perimeter $i$, and consider the expected time $T$ it takes for this semi-perimeter to grow by one. We have to wait for a 1-neighbour infection on one of the sides of the droplet, and after this it is sufficient to wait for at most $i$ further 2-neighbour infections, so 
\begin{equation}\label{eq:sketch:Tapprox}
\frac{n}{2ki}  \, \le \, T \, \le \, \frac{n}{2ki} + i.
\end{equation}
Being somewhat more precise, if we let $Y_1,Y_2,\dots$ and $Z_1,Z_2,\dots$ be independent random variables such that\footnote{In this paper we write $X\sim\Exp(\lambda)$ to mean $X$ has the exponential distribution with mean $\lambda$.} $Y_i \sim \Exp\left(n/2ki\right)$ and $Z_i \sim \Exp(1)$ for each $i \in \N$, then there exist natural couplings in which we have
$$ \sum_{i=2}^{\Omega(m)} Y_i  \, \le \, T^{-}(m) \, \le \, T^{+}(m) \, \le \,  \sum_{i=2}^{O(m)} Y_i + \sum_{i=1}^{O(m^2)} Z_i.$$
Now, if $m \gg 1$, then
$$\Ex\bigg[ \sum_{i=2}^{\Theta(m)} Y_i \bigg] \, = \, \big( 1 + o(1) \big) \frac{n}{2k} \log{m} \qquad \text{and} \qquad \Ex\bigg[ \sum_{i=2}^{O(m^2)} Z_i \bigg] \, = \, O(m^2),$$
and if $m \ll \sqrt{(n / k) \log (n/k)}$ then $m^2 \ll \frac{n}{k} \log{m}$. To prove Theorem~\ref{thm:acc}, it therefore suffices to prove a suitable concentration inequality for the sum $\sum_{i=2}^m Y_i$. We will do so using the following special case of Freedman's concentration inequality~\cite{F75}. 

\begin{lemma}[Freedman's inequality]\label{lem:Freedman}
Let $a,C>0$ and let $X_1,\dots,X_m$ be independent random variables such that $\E [X_i] \leq 0$ and $|X_i|\leq C$ almost surely for each $1\leq i\leq m$. Then, for every $a > 0$, 
\[
\P\bigg( \sum_{i=1}^m X_i \geq a \bigg) \leq \exp\left(-\frac{a^2}{2\sum_{i=1}^m \Var(X_i) + 2Ca}\right).
\]
\end{lemma}

When $m \gg \sqrt{n/k}$, the bounds in~\eqref{eq:sketch:Tapprox} become rather weak, and the nature of the growth changes. Instead of waiting for a single nucleation on each side of the droplet, it becomes more efficient to wait for time $\Theta(\sqrt{n/k})$, during which time 1-neighbour infections appear with density $\Theta(\sqrt{k/n})$, and then wait for the same amount of time for the space in between to be filled in. Although not all sites on the boundary of the droplet will be infected within time $O(\sqrt{n/k})$, it is not too hard to get around this problem, as long as we are willing to give away a multiplicative constant, see Lemmas~\ref{lem:tv:upper} and~\ref{lem:tv:lower}. 


\subsection{Bootstrap percolation}\label{sec:boot:intro}

The main difference between the Kesten--Schonmann model, discussed in the previous subsection, and the Dehghanpour--Schonmann model, studied in this paper, is that in the latter different droplets interact via a process known as `bootstrap percolation'. This is a particularly simple and well-studied deterministic cellular automaton (that is, a discrete dynamic system whose update rule is homogeneous and local) which has previously found applications in the study of the Glauber dynamics of the Ising model~\cite{FSS,M11}. We will use some standard techniques from the area to prove  Theorem~\ref{thm:NG2d}.

The classical `$r$-neighbour' bootstrap process, which was introduced in 1979 by Chalupa, Leath and Reich~\cite{CLR}, is as follows. Given a graph $G$, an integer $r \in \N$ and a set $A \subset V(G)$, we set $A^{(0)} = A$ and
$$A^{(i+1)} \, = \, A^{(i)} \cup \big\{ v \in V(G) \,:\, |N(v) \cap A^{(i)}| \ge r \big\}$$
for each $i \ge 0$. We write $\< A \>_r = \bigcup_i A^{(i)}$ for the set of eventually-infected vertices, and say that $A$ \emph{percolates} under the $r$-neighbour process if $\< A \>_r = V(G)$. The main parameter of interest is the critical probability
\begin{equation}\label{def:pc:r:boot}
p_c(G,r) \, = \, \inf\Big\{ p \in (0,1) \,:\, \Pr_p\big( \< A \>_r = V(G) \big) \ge 1/2 \Big\},
\end{equation}
where $\Pr_p$ denotes that $A$ is a $p$-random subset of $V(G)$, that is, each vertex $x \in V(G)$ is an element of $A$ with probability $p$, all independently of one another. For the lattice $\Z^d$, it was shown by Schonmann~\cite{Sch} that $p_c(\Z^d,r) = 0$ if $r \le d$ and $p_c(\Z^d,r) = 1$ otherwise. Much more precise bounds for finite grids (or tori) were proved in~\cite{AL,BBM,CC,CM,Hol}, culminating in the work of Balogh, Bollob\'as, Duminil-Copin and Morris~\cite{BBDM}, who proved that
\begin{equation*}\label{eq:BBDM:thm}
p_c\big( [n]^d,r \big) \, = \, \bigg( \frac{\lambda(d,r) + o(1)}{\log_{(r-1)} (n)} \bigg)^{d-r+1}
\end{equation*}
as $n \to \infty$ for every fixed $d \ge r \ge 2$, where $\lambda(d,r) > 0$ is an explicit constant, and the function $\log_{(r)}$ denotes an $r$-times iterated logarithm. In the case $d = r = 2$ this result was first proved by Holroyd~\cite{Hol}, who moreover showed that $\lambda(2,2) = \pi^2/18$, so 
\begin{equation}\label{eq:Hol:thm}
p_c\big( [n]^2,2 \big) \, = \, \bigg( \frac{\pi^2}{18} + o(1) \bigg) \frac{1}{\log n}.
\end{equation}
The upper bound in regime~$(a)$ of Theorem~\ref{thm:NG2d} follows easily from~\eqref{eq:Hol:thm} (see Proposition~\ref{prop:a:upper}), and to prove the corresponding lower bound we will adapt the proof from~\cite{Hol}, see Section~\ref{sec:lower:boot}. 

We will use various basic techniques from the study of bootstrap percolation in our analysis of the nucleation and growth process. Perhaps the most important of these is the so-called `rectangles process', which was introduced over 25 years ago by Aizenman and Lebowitz~\cite{AL}. Since modifications of this process will play a key role in several of the proofs below, let us briefly describe it in the setting in which it was first used: the 2-neighbour bootstrap process on $\Z^2$.

\begin{defn}[The rectangles process]
Let $A = \{x_1,\ldots,x_m\}$ be a finite set of sites in $\Z^2$, and consider the collection $\big\{ (R_1,A_1), \ldots, (R_m,A_m)\big\}$, where $A_j = \{x_j\}$ and $R_j = \< A_j \>_2$ for each $j \in [m]$. Now repeat the following steps until STOP:
\begin{itemize}
\item[1.] If there exist two rectangles $R_i$ and $R_j$ in the current collection at distance at most two from one another, then choose such a pair, remove them from the collection, and replace them by $(\< A_i \cup A_j \>_2, A_i \cup A_j)$. 
\item[2.] If there do not exist such a pair of rectangles, then STOP.
\end{itemize}
\end{defn}

It is easy to see that (the union of) the final collection of rectangles is exactly the closure $\< A \>_2$ under the 2-neighbour bootstrap process on $\Z^2$. Moreover, at each step of the process, the current collection of rectangles are \emph{disjointly internally spanned} by $A$; that is, we have $R_j = \< A_j \>_2$ for each $j$, and the sets $A_j$ are disjoint. 

The rectangles process also proves the following key lemma. Let $\phi(R)$ denote the semi-perimeter of the rectangle $R$. 

\begin{AL}
If $R = \< A \>_2$ is a rectangle, then for every $1 \leq \ell \leq \phi(R)$, there exists a rectangle $R' \subset R$ with 
$$\ell \leq \phi(R') \leq 2\ell,$$
such that $\< R' \cap A \>_2 = R'$. 
\end{AL} 

\begin{proof}
At each step of the rectangles process, the maximum semi-perimeter of a rectangle in the current collection at most doubles. 
\end{proof}

In order to prove the lower bounds in Theorem~\ref{thm:NG2d}, we will couple the nucleation and growth process with two different variants of the rectangles process, and prove corresponding Aizenman--Lebowitz-type lemmas, see Lemmas~\ref{le:AL:RRP} and~\ref{le:AL:TV}. In order to define the various random variables involved in our couplings, we will find the following notation useful.

\begin{defn}
Given $S \subset \Z^2$, we denote by $[S]_t$ the set of sites infected at time $t$ under the Kesten--Schonmann process with initially infected set $S$. That is, if we infect the set $S$ at time zero, and then only allow 1- and 2-neighbour infections.   
\end{defn}

We will use the following lemma, which can be easily proved using the methods of~\cite{Hol} (or indeed the earlier methods of~\cite{AL}). For much more precise results on the time of bootstrap percolation, see~\cite{BBS}.

\begin{lemma}\label{lem:boot:upper:weak}
There exists a constant $C > 0$ such that the following holds for every $n \in \N$ and $1 \le k \le n$. Let $M > m > 0$, and suppose $[M]^2$ is partitioned into $m \times m$ squares in the obvious way. Let $A$ consist of the union of a $p$-random subset of the $m \times m$ squares. If
$$p \, \ge \, \frac{C}{\log (M/m)},$$
then $[A]_{M (\log (M/m))^4} = [M]^2$ with high probability as $M/m \to \infty$. 
\end{lemma}

We will use Lemma~\ref{lem:boot:upper:weak} to prove the upper bounds in regimes~$(b)$ and~$(c)$ of Theorem~\ref{thm:NG2d}. Finally, let us observe that we may  always restrict our attention to the square $S(n)$. 

\begin{lemma}\label{lem:Sn}
With high probability as $t \to \infty$, there is no path of infections from outside the square $S(t)$ to the origin in time $o(t)$. 
\end{lemma}

\begin{proof}
We simply count the expected number of such paths, and apply Markov's inequality. To spell it out, for each $m \ge t$ there are at most $3^{m+1}$ paths of length $m$ from outside $S(t)$ to the origin, and each contains at least $m/2$ steps of length $o(1)$ with probability $o(1)^m$. 
\end{proof}

Since the bounds in Theorem~\ref{thm:NG2d} are all $o(n)$, it follows that we may always set all rates outside $S(n)$ equal to zero.

\subsection{A sketch of the proof of Theorem~\ref{thm:NG2d}}

Having introduced our main tools, let us outline how they imply the bounds on $\tau$ in our main theorem. In the sketches below, let us write $\eps$ (resp. $C$) for an arbitrarily small (resp. large) positive constant.\medskip

\noindent \textbf{Regime $(a)$}: As noted above, it is easy to deduce the upper bound (which in fact holds for all $1 \le k \le n$) from Holroyd's theorem~\eqref{eq:Hol:thm}; to prove the lower bound when $k \ll \log n$ we repeat\footnote{In fact this is an oversimplification, since there are a number of additional technical complications to overcome in order to obtain the claimed bound, see Section~\ref{sec:lower:boot}.} the proof of Holroyd~\cite{Hol}, showing that at each step of the hierarchy the effect of the 1-neighbour infections is negligible. The basic idea is that our droplet(s) will meet various `double gaps', which it will take them some time to cross (using 1-neighbour infections). More precisely, each such crossing takes time roughly $n / (k \log n)$, and so the probability that we cross $\Omega(\log n)$ such double gaps is polynomially small in $n$. We therefore have at most ${\log n \choose o(\log n)} = n^{o(1)}$ choices for the positions of our double gaps, and this allows us to use the union bound. The main difficulty lies in finding a set of $\Omega(\log n)$ double gaps that must be crossed in a certain order, which allows us to couple the total time taken with a sum of exponential random variables.

\medskip

\noindent \textbf{Regime $(b)$}: To prove the upper bound in this regime, we choose $m$ so that $\eps m^2 \log m = C k / \log n$, partition $S(n^{1/3})$ into translates of $S(m)$, and observe that in each, the probability that there is at least one nucleation by time $\eps (n/k) \log m$ is roughly $C / \log n$. Moreover, if a copy of $S(m)$ contains a nucleation, then (by Theorem~\ref{thm:acc}) with high probability it is entirely infected by time $(1/2 + \eps) (n/k) \log m$. Finally, by Lemma~\ref{lem:boot:upper:weak}, the entirely infected copies of $S(m)$ will (bootstrap) percolate in $S(n^{1/3})$, and this complete infection occurs in time $o(n/k)$. 

To prove the lower bound, we restrict to $S(n)$ (since nucleations outside this box do not have time to reach the origin), and perform the following coupling: first run for time $t = (1/2 - \eps) (n/k) \log m$ adding only nucleations; then re-run time, adding 1-neighbour infections as they occur, and $2$-neighbour infections instantaneously, according to the rectangles process described above. By the Aizenman--Lebowitz lemma, there are three possibilities:
\begin{itemize}
\item[$(i)$] there exists a nucleation within distance $o(\sqrt{n/t})$ of the origin, \smallskip
\item[$(ii)$] there exists,  within distance $\sqrt{k \log n}$ of the origin, a rectangle $R$ of semi-perimeter roughly $\sqrt{n/t}$ that is internally filled by time $t$.\smallskip
\item[$(iii)$] there exists in $S(n)$ a rectangle $R$ of semi-perimeter roughly $\sqrt{k \log n}$ that is internally filled by time $t$. 
\end{itemize}
We use Markov's inequality to show that each of these possibilities is unlikely.

To be slightly more precise, we will first bound the number of nucleations in $R$, and then bound the probability that fewer nucleations grow to fill $R$ by time $t$. A key observation is that, since $2$-neighbour infections do not increase the perimeter of the droplet, the rate of growth can be bounded as in Theorem~\ref{thm:acc}. Our bound on the probability (with the maximum allowed number of nucleations) is just strong enough to beat the number of choices for the rectangle $R$ (i.e., in case $(ii)$ it is super-polynomial in $k$, and in case $(iii)$ it is super-polynomial in $n$). 

One interesting subtlety in the proof is that if $k$ is close to the upper bound (more precisely, if $\sqrt{n} (\log n)^{3/2} \ll k \ll \sqrt{n} (\log n)^{2}$) the droplets are typically already growing at terminal velocity when they start to combine in the bootstrap phase. However, this only causes problems in the proof of the upper bound.

\medskip

\noindent \textbf{Regime $(c)$}: The proof of the upper bound is similar to that in regime $(b)$. Indeed, we will set $t = \big(\frac{n^2}{k \log (n/k)} \big)^{1/3}$, $m = t \sqrt{k/n}$ and $M = m \cdot (n/k)^{1/4}$, and partition $S(M)$ into translates of $S(m)$. The probability that there is at least one nucleation in $S(m)$ by time $Ct$ is at least $C / \log (n/k)$, each such nucleation grows to fill its translate of $S(m)$ in time $Ct$ (see Lemma~\ref{lem:tv:upper}), and by Lemma~\ref{lem:boot:upper:weak} these entirely infected copies of $S(m)$ bootstrap percolate to infect in $S(M)$ in time $o(t)$. 

To prove the lower bound, the main step is to define a `generous rectangles process', and to show that with high probability this process does indeed contain the original process (see Lemmas~\ref{lem:tv:lower} and~\ref{lem:TV:coupled}). Using this coupling we can prove an Aizenman--Lebowitz-type lemma, and then (modulo some non-trivial technical differences) repeat the proof described above. 


\section{Accelerating regime}\label{sec:AR}

In this section we will prove the following two lemmas, which give lower and upper bounds respectively on the rate of growth of a single droplet. First, let $T_m$ be the random time it takes in the Kesten--Schonmann process for a single nucleation at the origin to grow to contain a rectangle of semi-perimeter $m$. That is, define
$$T_m \, = \, \inf \big\{ t \ge 0 : \exists \, R \subset [\0]_t \text{ with } \phi(R) \ge m \big\}$$
for each $m \in \N$. The results of this section hold for every function $1 \le k = k(n) \le n$. 

\begin{lemma}\label{lem:acc:upper}
Fix $\delta > 0$ sufficiently small, and let $n \in \N$ be sufficiently large. If $m \in \N$ satisfies 
$$m \, \leq \, \delta\sqrt{\frac{n}{k} \log m},$$ 
then
\[
\P\left( T_m \geq (1+\delta)\frac{n}{2k}\log m \right) \leq \exp\Big( -\delta^2 \sqrt{\log m}\Big).
\]
\end{lemma}

For the lower bound we will need to understand the action on arbitrary (finite) initial configurations of a `modified' Kesten--Schonmann process, in which all $2$-neighbour infections are instantaneous.\footnote{To be precise, in this process each vertex $v \in \Z^2$ is infected at rate $0$ if it has no infected neighbours, at rate $k/n$ if it has one infected neighbour, and at rate $\infty$ if it has two or more infected neighbours.} Let us denote by $T^*_m(A)$ the random time at which the set of infected vertices in this process first has total semi-perimeter at least $m$, if $A$ is the set of infected vertices at time zero.

\begin{lemma}\label{lem:acc:lower}
Fix $\delta > 0$ sufficiently small, and let $m \in \N$ be sufficiently large. For every $1 \le \ell \le e^{-1/\delta^2} m$, and every set $A \subset \Z^2$ of size at most $\ell$, 
\[
\P\bigg( T^*_m(A) \leq (1-\delta)\frac{n}{2k}\log\frac{m}{\ell} \bigg) \, \leq \, \exp\bigg(- \delta^2 \max\big\{ m^{\delta/2}, \ell \big\} \bigg).
\]
\end{lemma}

We will prove Lemmas~\ref{lem:acc:upper} and~\ref{lem:acc:lower} using Freedman's inequality, Lemma~\ref{lem:Freedman}. 


\subsection{Upper bounds on the growth of a droplet}\label{sec:AC:upper}

We begin the proof of Lemma~\ref{lem:acc:upper} by making a simple, but key, observation. Let $Y_1,Y_2,\dots$ and $Z_1,Z_2,\dots$ be independent random variables such that
\[
Y_i \sim \Exp\left(\frac{n}{2ki}\right) \qquad \text{and} \qquad Z_i \sim \Exp(1)
\]
for each $i \in \N$.

\begin{obs}\label{obs:basiccoupling}
There exists a coupling such that
\[
T_m \leq \sum_{i=2}^m Y_i + \sum_{i=1}^{m^2} Z_i
\]
for every $m \in \N$.
\end{obs}

\begin{proof}
Let $R$ be a rectangle with semi-perimeter $i$. The time it takes for a new 1-neighbour infection to arrive on the side of $R$ is an exponential random variable with mean $n / 2ki$, and the time for the infection to spread along the side of $R$ is bounded above by the sum of $i$ independent exponentially distributed random variables with mean $1$. Since the new rectangle $R' \supset R$ thus formed has semi-perimeter $i+1$, the observation follows.
\end{proof}

We need two more standard bounds, which we will use several times.

\begin{lemma}\label{lem:poisson}
For every $\eps > 0$ there exists $\delta > 0$ such that the following holds for every $\lambda > 0$ and all sufficiently large $s \in \N$. Let $X_1,\dots,X_s$ be independent $\Exp(\lambda)$ random variables.
\begin{itemize}
\item[$(a)$] If $\lambda s \leq (1-\eps) t$, then
$$\P\bigg[ \sum_{i=1}^s X_i \geq t \bigg] \leq e^{-\delta s}.$$
\item[$(b)$] If $\lambda s \geq e^2 t$, then
$$\P\bigg[ \sum_{i=1}^s X_i \leq t \bigg] \leq \bigg( \frac{et}{\lambda s} \bigg)^{s}.$$
\end{itemize}
\end{lemma}

\begin{proof}
Writing $\textup{Po}(\mu)$ for a Poisson random variable with mean $\mu$, both inequalities follow easily from the fact that $\Pr\big( \sum_{i = 1}^s \Exp(\lambda) \le t \big) = \Pr\big( \textup{Po}(t/\lambda) \ge s \big)$.
\end{proof}

We can now prove Lemma~\ref{lem:acc:upper}.

\begin{proof}[Proof of Lemma~\ref{lem:acc:upper}]
By Observation~\ref{obs:basiccoupling}, it suffices to prove that
\[
\P\bigg( \sum_{i=2}^m Y_i + \sum_{i=1}^{m^2} Z_i \geq (1+\delta)\frac{n}{2k}\log m \bigg) \leq \exp\Big( -\delta^2\sqrt{\log m}\Big).
\]
Moreover, by Lemma~\ref{lem:poisson}, and since $m^2 \leq (\delta n / 5k) \log m$, we have
\[
\P\bigg( \sum_{i=1}^{m^2} Z_i > \frac{\delta n}{4k} \log m \bigg) \leq e^{-\Omega(m^2)},
\]
so in fact it will suffice to prove that
\begin{equation}\label{eq:acc:upper:suff}
\P\bigg( \sum_{i=2}^m Y_i \geq \bigg( 1 + \frac{\delta}{2} \bigg) \frac{n}{2k}\log m \bigg) \leq \exp\Big( -2\delta^2 \sqrt{\log m} \Big).
\end{equation}
To prove~\eqref{eq:acc:upper:suff}, we will use Freedman's inequality. In order to do so, we need to define a sequence of independent random variables $X_1,\ldots,X_m$ with $\E[X_i] \le 0$ and $|X_i|$ bounded above for each $i \in [m]$, so set
\begin{equation}\label{def:Xi:accupper}
X_i = \min\left\{ Y_i - \E [Y_i], \, \frac{n}{k}\sqrt{\log m} \right\}
\end{equation}
for each $i \in [m]$. Since $\sum_{i=2}^m \E[Y_i] \le (n/2k) \log m$, we have
\begin{multline}\label{eq:Yismallupper}
\P\bigg( \sum_{i=2}^m Y_i \geq \bigg( 1 + \frac{\delta}{2} \bigg) \frac{n}{2k}\log m \bigg) \\
\leq \P\bigg( \sum_{i=2}^m X_i \geq \frac{\delta n}{4k}\log m \bigg) + \P\bigg( Y_i \geq \frac{n}{k}\sqrt{\log m} \, \text{ for some } i \in [m] \bigg),
\end{multline}
so it suffices to bound the two terms on the right. For the first, Freedman's inequality with $C = (n/k)\sqrt{\log m}$ and $a=(\delta n/4k)\log m$ implies that
\begin{align}
\P\bigg( \sum_{i=2}^m X_i \geq \frac{\delta n}{4k}\log m \bigg) &\leq \exp\left(-\frac{(\delta n/4k)^2 (\log m)^2}{(n/k)^2 + (\delta/2) (n/k)^2 (\log m)^{3/2}} \right) \notag \\
& \leq \exp\left(-\frac{\delta \sqrt{\log m}}{9}\right), \label{eq:Yismallupper1}
\end{align}
since
\[
\sum_{i=2}^m \Var (X_i) \leq \sum_{i=2}^m \Var (Y_i) = \sum_{i=2}^m \left(\frac{n}{2ki}\right)^2 \leq \frac{n^2}{2k^2}.
\]
For the second, the union bound implies that 
\begin{equation}\label{eq:Yismallupper2}
\P\bigg( \bigcup_{i = 1}^m \Big\{ Y_i \geq \frac{n}{k}\sqrt{\log m} \Big\} \bigg) \leq \sum_{i=2}^m \exp\Big(-2i\sqrt{\log m}\Big) \leq \exp\big(-\sqrt{\log m}\big). 
\end{equation}
Combining \eqref{eq:Yismallupper1} and \eqref{eq:Yismallupper2} with \eqref{eq:Yismallupper}, and recalling that $\delta$ is sufficiently small, gives
\[
\P\bigg( \sum_{i=2}^m Y_i \geq \bigg( 1 + \frac{\delta}{2} \bigg) \frac{n}{2k}\log m \bigg) \leq \exp\left(-\frac{\delta \sqrt{\log m}}{10}\right)  \leq \exp\Big( -2\delta^2 \sqrt{\log m} \Big),
\]
which proves \eqref{eq:acc:upper:suff}, and hence the lemma.
\end{proof}

\pagebreak

\subsection{Lower bounds on the growth of a droplet}\label{sec:AC:lower}

Recall that $Y_1,Y_2,\dots$ are independent random variables with $Y_i \sim \Exp\big( n / 2ki \big)$ for each $i \in \N$. We begin with a simple but key observation, cf. Observation~\ref{obs:basiccoupling}.

\begin{obs}\label{obs:basiccoupling:lower}
For every $A \subset S(n)$ with $|A| \le \ell$, there exists a coupling such that
\[
T^*_m(A) \geq \sum_{i=2\ell}^m Y_i
\]
for every $m \in \N$.
\end{obs}

\begin{proof}
Observe that the total semi-perimeter of the infected squares is initially at most $2\ell$, and is not increased by 2-neighbour infections. Moreover, the number of sites that can be infected via a 1-neighbour infection is exactly the total perimeter of the currently infected sites (since 2-neighbour infections are instantaneous). Therefore, the time taken for the total semi-perimeter to increase from $i$ to $i+1$ can be coupled with an exponential random variable with mean $n / 2ki$, as required.
\end{proof}

The proof of Lemma~\ref{lem:acc:lower} is similar to that of Lemma~\ref{lem:acc:upper}. In particular, we will again use Freedman's inequality, but in order to obtain the required super-polynomial bound we will need to be somewhat more careful. 

\begin{proof}[Proof of Lemma~\ref{lem:acc:lower}]
Recall that $A\subset S(n)$ has size at most $\ell\leq e^{-1/\delta^2} m$. To begin, set $\ell' = \max\{ 2\ell, m^{\delta/2} \}$, and note that, by Observation~\ref{obs:basiccoupling:lower}, it suffices to prove that
\[
\P\left( \sum_{i = \ell'}^m Y_i \leq (1-\delta)\frac{n}{2k}\log\frac{m}{\ell} \right) \leq e^{-\delta^2 \ell'}.
\]
Now, with Freedman's inequality in mind, set
$$X_i = \min\Big\{ Y_i - \E [Y_i], \, \frac{n}{k\ell'} \log \frac{m}{\ell} \Big\}$$
for each $i \in \N$, and observe that 
$$\sum_{i=\ell'}^m X_i \, \le \sum_{i=\ell'}^m Y_i - \sum_{i=\ell'}^m \E[Y_i] \, \le \, \sum_{i=\ell'}^m Y_i - \bigg( 1 - \frac{\delta}{2} \bigg) \frac{n}{2k}\log\frac{m}{\ell},$$
where the final inequality follows from
\begin{equation}\label{eq:ellprime}
\sum_{i=\ell'}^m \E[Y_i] \geq \frac{n}{2k}\log\frac{m}{\ell'} \geq \left(1-\frac{\delta}{2}\right)\frac{n}{2k}\log\frac{m}{\ell}.
\end{equation}
It will therefore suffice to prove that
\begin{equation}\label{eq:acc:lower:suff}
\P\left( \sum_{i=2\ell'}^m X_i \leq - \frac{\delta n}{4k}\log\frac{m}{\ell} \right) \leq e^{-\delta^2 \ell'}.
\end{equation}
To prove~\eqref{eq:acc:lower:suff}, we will apply Freedman's inequality to the sequence $-X_{2\ell'}',\ldots,-X'_m$, where 
$$X_i' = X_i - \E [X_i].$$
Note that $\Ex[X'_i] = 0$, and that  
$$- \frac{n}{4k \ell'} \, \le \, X_i' \, \le \, \frac{n}{k\ell'} \log \frac{m}{\ell}$$
for every $i \ge 2\ell'$, since
\[
-\frac{n}{4k\ell'} \le -\frac{n}{2ki} = -\E[Y_i] \le X_i \le \frac{n}{2k\ell'} \log \frac{m}{\ell}
\]
and $-n/4k\ell' \le \Ex[X_i] \le 0$. Furthermore, $\Var(X_i') = \Var(X_i) \leq \Var(Y_i)$, which implies that
\[
\sum_{i=2\ell'}^m \Var(X_i') \leq \sum_{i=2\ell'}^m \left(\frac{n}{2ki}\right)^2 \leq \frac{n^2}{k^2 \ell'}.
\]
Therefore, applying Freedman's inequality, we obtain
\begin{align}
\P\left( \sum_{i=2\ell'}^m X'_i \leq - \frac{\delta n}{8k}\log\frac{m}{\ell} \right) & \leq \exp\left(-\frac{(\delta n / 8k)^2 (\log m/\ell)^2}{n^2 / k^2 \ell' + (2n / k \ell')(\delta n / 8k)(\log m/\ell)^2} \right) \nonumber\\
&\leq \exp\left(- \delta \ell' / 32 \right), \label{eq:Xprimebound}
\end{align}
since $\log m/\ell \geq 1/\delta$ and $m$ is sufficiently large. Finally, note that
\begin{align*}
\sum_{i = 2\ell'}^m \E[X_i] & \, \ge \, - \sum_{i = 2\ell'}^m \Pr\bigg( Y_i \ge \frac{n}{2k\ell'} \log \frac{m}{\ell} \bigg) \Ex[Y_i] \\
& \, \ge \, - \sum_{i = 2\ell'}^m \frac{\ell}{m} \cdot \frac{n}{2ki} \, \ge \, - \frac{\ell}{m} \cdot \frac{n}{2k} \log \frac{m}{\ell'} \, \ge \, - \frac{\delta n}{8k} \log\frac{m}{\ell},
\end{align*}
where the final inequality follows from the definition of $\ell'$ (cf.~\eqref{eq:ellprime}), and since $\delta$ was chosen sufficiently small, so $\ell \le e^{-1/\delta^2} m \le \delta m/8$. It follows easily that
$$\sum_{i=2\ell'}^m X'_i \, = \, \sum_{i=2\ell'}^m X_i - \sum_{i=2\ell'}^m \E[X_i] \, \le \, \sum_{i=2\ell'}^m X_i  + \frac{\delta n}{8k}\log\frac{m}{\ell},$$
and so~\eqref{eq:Xprimebound} implies~\eqref{eq:acc:lower:suff}, which completes the proof of the lemma.  
\end{proof}

Finally, let us note the following two easy lemmas, which bound the probability that a rectangle contains too many nucleations. Let $P_t(R,\ell)$ denote the probability that exactly $\ell$ nucleations occur in $R$ by time $t$. We will use the following bounds in conjunction with Lemma~\ref{lem:acc:lower}. 

\begin{lemma}\label{lem:countingnucleations}
Let $R \subset S(n)$ be a rectangle of semi-perimeter $m$, and suppose that $t \le \frac{n}{4k}\log \big( \frac{k}{\log n} \big)$. Then
$$P_t(R,\ell) \, \le \, \bigg( \frac{m^2 \log \big( \frac{k}{\log n} \big)}{4\ell k} \bigg)^\ell.$$
In particular, if $\ell \ge (m^2/k) \log \big( \frac{k}{\log n} \big)$ and $m \ge \sqrt{k \log n}$, then 
$$P_t(R,\ell) \le n^{-\log (\frac{k}{\log n})},$$
and if $m \le \sqrt{k / \log \big( \frac{k}{\log n} \big)}$ and $\ell \ge \log k$, then
$$P_t(R,\ell) \le k^{-\log\log k}.$$
\end{lemma}

\begin{proof}
Note first that 
\begin{equation}\label{eq:PtRell}
P_t(R,\ell) \, \le \, {m^2/4 \choose \ell} \bigg( \frac{t}{n} \bigg)^\ell \, \le \, \bigg( \frac{em^2 t}{4\ell n} \bigg)^\ell \, \le \, \bigg( \frac{m^2 \log \big( \frac{k}{\log n} \big)}{4\ell k} \bigg)^\ell,
\end{equation}
since $t \le \frac{n}{4k}\log k$. It follows that, if $\ell \ge (m^2/k) \log \big( \frac{k}{\log n} \big)$ and $m \ge \sqrt{k \log n}$, then
$$P_t(R,\ell) \, \le \, \bigg( \frac{m^2 \log \big( \frac{k}{\log n} \big)}{4\ell k} \bigg)^\ell \, \le \, \exp\bigg( - \frac{m^2}{k} \log \left( \frac{k}{\log n} \right) \bigg) \, \le \, n^{-\log (\frac{k}{\log n})},$$
and if 
$m \le \sqrt{k / \log \big( \frac{k}{\log n} \big)}$ and $\ell \ge \log k$, then
$$P_t(R,\ell) \, \le \, \bigg( \frac{m^2 \log \big( \frac{k}{\log n} \big)}{4\ell k} \bigg)^\ell  \, \le \, \bigg( \frac{1}{\log k} \bigg)^{\log k} \, \le \, k^{-\log\log k},$$
as required.
\end{proof}

We will use the following variant in Section~\ref{sec:lowerTV}.

\begin{lemma}\label{lem:TVnucleations}
There exists $c > 0$ such that the following holds. Set 
$$t = c \cdot\bigg(\frac{n^2}{k\log (n/k)} \bigg)^{1/3}, \quad m = \big( kn \log (n/k) \big)^{1/6} \quad \text{and} \quad M = m \sqrt{\log (n/k)}.$$
If $k \ll n$, then
$$P_t\big( S(M),\log(n/k) \big) \ll \left(\frac{k}{n}\right)^{4} \quad \text{and} \quad P_t\big(S(m), \log(n/k)^{1/3} \big) \ll \left( \frac{1}{\log (n/k)} \right)^{4}.$$
\end{lemma}

\begin{proof}
Note first that, as in~\eqref{eq:PtRell}, we have
$$P_t(S(M),\ell) \, \le \, {5M^2 \choose \ell} \bigg( \frac{t}{n} \bigg)^\ell \le \, \bigg( \frac{(5ec)^3}{\ell^3} \cdot \frac{M^6}{k n \log (n/k) } \bigg)^{\ell/3} \le \, \bigg( \frac{5ec \log (n/k)}{\ell} \bigg)^{\ell},$$
since $t = c \cdot \big(\frac{n^2}{k\log (n/k)} \big)^{1/3}$ and $M = (kn)^{1/6} \big( \log (n/k) \big)^{2/3}$. It follows that, if $\ell = \log (n/k)$ and $c > 0$ is sufficiently small, then
$$P_t(S(M),\ell) \, \ll \, e^{-4\ell} = \, \left(\frac{k}{n}\right)^{4},$$
as claimed. Similarly, if $\ell = \log (n/k)^{1/3}$ and $c \le 1$ (say), then
$$P_t(S(m),\ell) \, \le \, \bigg( \frac{(5ec)^3}{\ell^3} \cdot \frac{m^6}{k n \log (n/k) } \bigg)^{\ell/3} \, \ll \, 2^{-\ell} \ll \, \left( \frac{1}{\log (n/k)} \right)^{4},$$
as required.
\end{proof}

\section{Terminal velocity regime}\label{sec:TV}

In this section we will show how to control the growth of a droplet once it has reached its terminal velocity. We begin with the following lemma, which we will use to prove the upper bounds in Theorem~\ref{thm:NG2d} when $k \gg \log n$. 

\begin{lemma}\label{lem:tv:upper}
There exists a constant $C > 0$ such that the following holds. Let $m = m(n) \in \N$ be such that $1 \ll \log m \ll \sqrt{n/k}$ and let $R$ be a rectangle of semi-perimeter $2\min\big\{ m, \sqrt{n/k} \big\}$ such that $R \cap S(m)\neq\emptyset$. Then
\[
S(m) \subset [R]_{Ct}
\]
with high probability, where $t = m \sqrt{n / k}$.
\end{lemma}

We remark that the condition $\log m \ll \sqrt{n/k}$ is not necessary for the lemma to hold, but it simplifies the proof, and will always hold in our applications.

\begin{proof}
Set $s = \min\big\{ m, \sqrt{n/k} \big\}$, and note that at least one side of $R$ has length at least $s$: without loss of generality it is the horizontal side. We will show that, with high probability, $R$ grows vertically by at least $2m$ steps, and then horizontally by at least $2m$ steps, within time $C t$.

Let $R_0 = \{ (x,b) : a \le x < a + s \}$ be such that $R_0 \subset R$ and $R_0 \cap S(m) \ne \emptyset$, and define a new rectangle $R_1$ as follows:
$$R_1 = \big\{ (x,y) \,:\, a \le x < a + s \text{ and } b - 2m \le y \le b + 2m \big\}.$$
We first show that $R$ quickly grows to fill $R_1$.

\medskip
\noindent \textbf{Claim 1:} $R_1\subset [R_0]_{10t}$ with high probability.

\begin{proof}[Proof of claim]
Let $Y_1,Y_2,\dots$ be a sequence of independent $\Exp\big(\sqrt{n/k}\big)$ random variables, and let $Z_1,Z_2,\dots$ be an independent sequence of $\Exp(1)$ random variables. As in Observation~\ref{obs:basiccoupling}, there exists a coupling such that
\[
\P\Big( R_1\subset [R_0]_{10t} \Big) \geq \P\bigg( \sum_{i=1}^{4m} Y_i + \sum_{i=1}^{4sm} Z_i \leq 10t \bigg),
\]
where here we have used the fact that $s \ge \sqrt{n/k}$. But, noting that $sm \le t$ (and since $m\gg 1$), we have 
\[
\P\bigg( \sum_{i=1}^{4m} Y_i + \sum_{i=1}^{4sm} Z_i \geq 10t \bigg) \leq \P\bigg( \sum_{i=1}^{4m} Y_i \geq 5t \bigg) + \P\bigg( \sum_{i=1}^{4sm} Z_i \geq 5t \bigg) = o(1)
\]
by Lemma~\ref{lem:poisson}, so this proves the claim.
\end{proof}

Next, we define
$$R_2 = \big\{ (x,y) \,:\, a - 2m \le x < a + 2m + s \text{ and } b \le y < b + s \big\},$$
and show that $R_1$ quickly grows to fill $R_2$.

\pagebreak
\noindent \textbf{Claim 2:} $R_2 \subset [R_1]_{10t}$ with high probability.

\begin{proof}[Proof of claim]
Since $m \ge s$, this follows exactly as in the proof of Claim~1.
\end{proof}

Finally, we fill out the corners using 2-neighbour infections.

\medskip
\noindent \textbf{Claim 3:} $S(m)\subset [R_1 \cup R_2]_t$ with high probability.

\begin{proof}[Proof of claim]
Growing diagonal-by-diagonal, there exists a coupling such that
$$\max\big\{Z_1,\dots,Z_{16m^2}\big\} \leq \frac{t}{4m} \quad \Rightarrow \quad  S(m)\subset [R_1 \cup R_2]_t.$$
Since $\log m \ll \sqrt{n/k}$, we have
\[
\P\bigg( \max\big\{Z_1,\dots,Z_{16m^2}\big\} \leq \frac{t}{4m} \bigg) = \bigg(1-\exp\bigg(-\frac{t}{4m}\bigg)\bigg)^{16m^2} = 1-o(1),
\]
and so the claim follows.
\end{proof}

The lemma now follows from Claims~1,~2 and~3.
\end{proof}

For convenience, let us note the following immediate consequence of Lemmas~\ref{lem:acc:upper} and~\ref{lem:tv:upper}.

\begin{lemma}\label{lem:acc:upper:fill}
If $m = m(n) \in \N$ satisfies 
$$1 \, \ll \, m \, \ll \, \sqrt{\frac{n}{k} \log m},$$ 
then a single nucleation at the origin grows to contain the square $S(m)$ within time
$$\big( 1 + o(1) \big) \frac{n}{2k} \log m$$
with high probability as $n \to \infty$.
\end{lemma}

\begin{proof}
By Lemma~\ref{lem:acc:upper}, a single nucleation grows to completely infect a rectangle of semi-perimeter $2m$ within time $\big( 1 + o(1) \big) \frac{n}{2k} \log m$ with high probability. By Lemma~\ref{lem:tv:upper} this rectangle then grows to completely infect $S(m)$ within a further period of size $O\big( m\sqrt{n/k} \big)$. Since $m \ll \sqrt{n/k} \log m$, the lemma follows. 
\end{proof}


We now turn to the lower bound of Theorem~\ref{thm:NG2d} in regime $(c)$, which requires us to prove an upper bound on the terminal velocity of a droplet. We will use the following lemma in Section~\ref{sec:lowerTV}. 

\begin{lemma}\label{lem:tv:lower}
Let $R \subset S(n)$ be a rectangle, and let $R'$ be the rectangle obtained from $R$ by enlarging each of the sides\footnote{That is, the length of each side of $R'$ is $200n^{1/4}$ larger than that of the corresponding side of $R$, and $R'$ has the same centre as $R$.} by $200n^{1/4}$. If $t = n^{3/4} / k^{1/2}$, then
$$\Pr\big( [R]_t \not\subset R' \big) \, \le \, \frac{1}{n^3}$$
for all sufficiently large $n \in \N$.  
\end{lemma}

To prove Lemma~\ref{lem:tv:lower}, we will consider the following (more generous) process $X(t)$. Set $X(0) = \{ (x,y) \in \Z^2 : y \le 0 \}$, and thereafter let each infected site infect its neighbour directly above at rate $k/n$, its two horizontal neighbours at rate $1$, and the site below it instantly. Note that there is a subtle difference (in addition to the obvious differences) between the `generous' process $X(t)$ and the Kesten--Schonmann process. In the Kesten--Schonmann process, each uninfected site has an associated rate of infection, which depends on the number of its infected neighbours. In the generous process, each infected site has \emph{three associated Poisson processes} (one each for left, right and upwards infections), so the viewpoint is switched from sites becoming infected to sites causing infections. Nevertheless, it is straightforward to couple the generous process with the Kesten--Schonmann process with initially infected set $X(0)$. 
Next, define
$$X^*_m \, = \, \inf\big\{ t > 0 :  \exists \, (x,y) \in X(t) \text{ with } x \in \{-n,\dots,n\} \text { and } y = m \big\}$$
for each $m \in \N$. We will prove the following bound on $X^*_m$. 

\begin{lemma}\label{lem:Xstar:countingpaths}
If $k \ll n$, then
$$\Pr\bigg( X^*_m < \frac{m}{100} \sqrt{ \frac{n}{k} } \bigg) \, \le\, n^3 \cdot 2^{-m}$$
for every $m \in \N$.
\end{lemma}

\begin{proof}
Let $\Gamma_m$ denote the set of all paths $\gamma$ between the sets $\Z \times \{0\}$ and $\{-n,\ldots,n\}\times\{m\}$, and let us write $m'(\gamma)$ and $\ell(\gamma)$ for the number of downward vertical steps and the number of horizontal steps in $\gamma$, respectively. Given $\gamma \in \Gamma_m$, define $t(\gamma)$ to be the (random) total time associated with the sequence of infections encoded by $\gamma$.\footnote{Note that $t(\gamma)$ is well-defined, by the coupling with the Poisson processes described in the paragraph before Lemma~\ref{lem:Xstar:countingpaths}: for any $\gamma\in\Gamma_m$, there is a sequence of successive clock rings associated with (the directed edges along) that path, and $t(\gamma)$ is the time of the final clock ring. Thus it does not matter whether $\gamma$ encodes the \emph{actual} sequence of infections.} Since horizontal infections have rate $1$ and upward vertical infections have rate $k/n$, the distribution of $t(\gamma)$ is precisely
\begin{equation}\label{eq:tgamma:dist}
\sum_{i=1}^{m+m'(\gamma)} Y_i \, +\, \sum_{j=1}^{\ell(\gamma)} Z_j,
\end{equation}
where the random variables $Y_i \sim \Exp(n/k)$ and $Z_j \sim \Exp(1)$ are all independent. 

We claim that (for any $a > 0$) the event $\{ X^*_m < a \}$ is precisely the union over paths $\gamma \in \Gamma_m$ of the event that $t(\gamma) < a$. Indeed, suppose that $(x,m) \in \{-n,\dots,n\}\times \{m\}$ is infected by time $a$, and define an associated path $\gamma = (\gamma_0,\dots ,\gamma_{L})$ by setting $\gamma_0 = (x,m)$, and then defining $\gamma_{i+1}$ to be the site that infected $\gamma_i$ for each $0 \le i < L$, where $\gamma_L$ is the first site in the set $\Z\times \{0\}$. Note that $L=m+2m'+\ell$.
 
It remains to bound, for each triple $(m,m',\ell)$, the number of paths $\gamma\in\Gamma_m$ with $m'(\gamma) = m'$ and $\ell(\gamma) = \ell$, and the probability for each such $\gamma$ that $t(\gamma) < (m/100) \sqrt{ n/k }$. Recalling that the distribution of $t(\gamma)$ is given by~\eqref{eq:tgamma:dist}, observe first that 
$$\Pr\bigg( \sum_{i=1}^{m+m'} Y_i \le \frac{m}{100} \sqrt{ \frac{n}{k} } \bigg) \, \le \, \bigg( \frac{e}{100}\sqrt{ \frac{k}{n} } \bigg)^{m+m'},$$
by Lemma~\ref{lem:poisson}. Now, if $L \le (m+m')\sqrt{n/k}$, then the number of corresponding choices of $\gamma$ is at most
\[
(2n+1) \binom{L}{m+m'} 4^{m+m'}\, \le \, (2n+1) \big( 4e\sqrt{n/k} \big)^{m+m'},
\]
and hence the expected number of such paths with $t(\gamma) < (m/100) \sqrt{ n/k }$ is at most
$$(2n+1) \bigg( 4e\sqrt{ \frac{n}{k} } \bigg)^{m+m'} \bigg( \frac{e}{100}\sqrt{ \frac{k}{n} } \bigg)^{m+m'} \, \le \, n \cdot 3^{-(m+m')}.$$
Summing over $L \le (m+m')\sqrt{n/k}$ and then over $m' \ge 0$, it follows that the expected number of such `short' paths is at most $n^2 \cdot m \cdot 3^{-m}$.
 
To bound the expected number of paths with $L > (m+m')\sqrt{n/k}$, observe first that
$$\Pr\bigg( \sum_{j=1}^{L - (m+m')} Z_j \le \frac{m}{100} \sqrt{ \frac{n}{k} } \bigg) \, \le \, \bigg( \frac{e}{50} \bigg)^{-L/2} \, \le \, 4^{-L},$$
by Lemma~\ref{lem:poisson}, since $(m/100) \sqrt{ n/k } \le L/100$ and $L - (m+m') \ge L/2$. Thus, noting that there are at most $(2n+1) \cdot 3^L$ paths $\gamma \in \Gamma_m$ of length $L$, the expected number of such paths with $t(\gamma) < (m/100) \sqrt{ n/k }$, and with $m'$ and $L$ given, is at most $(2n+1)\cdot 3^L\cdot 4^{-L}$. Summing over $L>(m+m')\sqrt{n/k}$ and then over $m' \ge 0$, we have that the expected number of such `long' paths is at most $n \cdot 3^{-m}$. This completes the proof of the lemma.
\end{proof}

We can now easily deduce Lemma~\ref{lem:tv:lower}. 

\begin{proof}[Proof of Lemma~\ref{lem:tv:lower}]
Recall that $t = n^{3/4} / k^{1/2}$, and suppose that $[R]_t \not\subset R'$. Then in at least one of the four directions, an event corresponding to $X^*_m < t$ occurred, where $m = 100n^{1/4}$. Now, noting that $(m/100) \sqrt{ n/k } = n^{3/4} / k^{1/2}$, it follows by Lemma~\ref{lem:Xstar:countingpaths} that this has probability at most
$$\frac{4n^3}{2^m} \, \le \, \frac{4n^3}{2^{n^{1/4}}} \, \le \, \frac{1}{n^3},$$
as claimed.
\end{proof}

\section{The upper bounds}\label{upper:sec}

In this section we will deduce the various upper bounds in Theorem~\ref{thm:NG2d} from the results in Sections~\ref{sec:AR} and~\ref{sec:TV}. We begin with regime $(a)$, where the claimed bound is an easy consequence of Holroyd's theorem. 

\begin{proposition}\label{prop:a:upper}
For every $k \ge 1$, and every $\delta > 0$, we have
\begin{equation}\label{eq:prop:a:upper}
\tau \leq \left( \frac{\pi^2}{18} + \delta \right) \frac{n}{\log n}
\end{equation}
with high probability.
\end{proposition}

\begin{proof}
Let $\eps = \eps(\delta) > 0$ be sufficiently small, and set $M = n^{1-\eps}$. By Holroyd's theorem, the collection of nucleations which occur in the box $S(M)$ by time 
$$t \, = \, \left( \frac{\pi^2}{18} + \eps \right) \frac{n}{\log M}$$ 
percolate with high probability under the $2$-neighbour bootstrap rule. Moreover, by Lemma~\ref{lem:boot:upper:weak}, this `bootstrap phase' takes time at most $M^{1 + o(1)} \ll t$. Hence, with high probability the origin is infected by time $\big( 1 + o(1) \big) t$, as required. 
\end{proof}

We next move on to regime $(b)$, where we need to do a little more work. 

\begin{proposition}\label{prop:b:upper}
Let $\log n \ll k \ll \sqrt{n}(\log n)^2$, and let $\delta>0$. Then
\[
\tau \leq \left(\frac{1}{4} + \delta\right) \frac{n}{k} \log \bigg( \frac{k}{\log n} \bigg)
\]
with high probability.
\end{proposition}

\begin{proof}
Set
\[
M = n^{1/3}, \qquad m = \frac{1}{\delta} \cdot \sqrt{\frac{k}{\log n \log m}} \qquad \text{and} \qquad t = \frac{n}{2k}\log m.
\]
(Note that the implicit definition of $m$ has a unique solution.) We will couple the nucleation and growth process up to time $(1 + 4\delta)t$ as follows. First, for the entire period, set the rates for all sites outside $S(M)$ equal to zero. Inside $S(M)$ we do the following:
\begin{enumerate}
\item[1.] Run for time $\delta t$ allowing only nucleations.\smallskip
\item[2.] Run for time $(1+2\delta)t$ allowing each nucleation to grow via 1- and 2-neighbour infections (i.e., as in the Kesten-Schonmann process), but ignoring the interaction between different droplets. \smallskip
\item[3.] Run for time $\delta t$ allowing only 2-neighbour infections.
\end{enumerate}
It is easy to see that the process just described is indeed a coupling, in the sense that if $\tau'$ is the time the origin is infected in the coupled process, then $\tau' \ge \tau$. 

Now, partition $S(M)$ into translates of $S(m)$, and for each such square $S$ consider the event $E_S$ that there was a nucleation in $S$ which grew to infect the entire square by the end of step 2. We will show that $E_S$ has probability $\gg 1 / \log n$. Since these events are independent and $m \cdot n^{\Omega(1)} < M < t \cdot n^{-\Omega(1)}$, it follows by Lemma~\ref{lem:boot:upper:weak} that, with high probability, the whole of $S(M)$ is infected by time $(1 + 4\delta) t$. 

To spell out the details, note first that the probability that a given translate of $S(m)$ has received at least one nucleation by time $\delta t$ is at least
\begin{equation}\label{eq:b:upper:nuc}
1 - \exp\left( -\frac{4m^2 \cdot \delta t}{n} \right) = 1 - \exp\left( -\frac{2}{\delta\log n} \right) \geq \frac{C}{\log n},
\end{equation}
where $C$ is the constant in Lemma~\ref{lem:boot:upper:weak}, since $\delta$ is arbitrarily small. We claim that if there is a nucleation in a given translate $S$ of $S(m)$ in step 1, then with high probability $S$ is entirely infected by the end of step 2. This follows immediately from Lemma~\ref{lem:acc:upper:fill} if $k \ll \sqrt{n} (\log n)^{3/2}$, since in this case $1 \ll m^2 \ll (n/k) \log m$. 

On the other hand, note that $m \ge 2 \sqrt{n/k}$ for all $k \ge \sqrt{n} \log n$, and thus
$$(1 + \delta ) \frac{n}{2k} \log \Big( 2\sqrt{n/k} \Big) \, \le \, (1+\delta)t.$$ 
By Lemma~\ref{lem:acc:upper}, it follows that, with high probability, there exists a rectangle $R$ of semi-perimeter at least $2\sqrt{n/k}$ that intersects $S$, and is entirely infected by time $(1+\delta)t$ after the start of step 2. Now, by Lemma~\ref{lem:tv:upper}, with high probability this rectangle grows so that $S$ is entirely infected after a further time 
$$O\bigg( m\sqrt{\frac{n}{k}} \bigg) \, = \, O\bigg( \sqrt{\frac{n}{\log n\log m}} \bigg) \, \ll \, \frac{\delta n}{2k} \log m \, = \, \delta t,$$
where the last inequality follows since $k \ll \sqrt{n}(\log n)^2$.

Finally, since (in our coupling) each translate of $S(m)$ evolves independently until time $(1+3\delta)t$, and since $m \cdot n^{\Omega(1)} < M < t \cdot n^{-\Omega(1)}$, it follows from Lemma~\ref{lem:boot:upper:weak} that, with high probability, $S(M)$ is completely infected by time $(1 + 4\delta) t$, as required.
\end{proof}

Finally, let us prove the upper bound in regime $(c)$. The proof is almost identical to that of Proposition~\ref{prop:b:upper} after some adjustment of the parameters.

\begin{proposition}\label{prop:c:upper}
There exists a constant $C > 0$ such that the following holds. If $\sqrt{n}(\log n)^2 \le k \ll n$, then
\[
\tau \le C \cdot \left( \frac{n^2}{k\log (n/k)} \right)^{1/3}.
\]
with high probability.
\end{proposition}

\begin{proof}
Let $C > 0$ be a sufficiently large constant, and set 
\[
t = \left(\frac{n^2}{ k \log (n/k)}\right)^{1/3}, \qquad m = t \cdot \sqrt{k/n} \qquad \text{and} \qquad M = m \cdot (n/k)^{1/4}.
\]
Define a coupling (and the events $E_S$) as in the proof of Proposition~\ref{prop:b:upper}, except with each period of the coupling lasting time $C t$, and observe that a translate $S$ of $S(m)$ contains at least one nucleation by time $C t$ with probability at least
$$1 - \exp\left( - \frac{4m^2 \cdot C t}{n}\right) = 1 - \exp\left( - \frac{4C t^3k}{n^2} \right) \ge \frac{C}{\log (n/k)}.$$
The proof now follows exactly as before: by Lemmas~\ref{lem:acc:upper} and~\ref{lem:tv:upper}, a nucleation in $S$ grows, with high probability, to fill the whole of $S$ within time
$$\big(1 + o(1) \big) \frac{n}{k} \log \Big( 2\sqrt{n/k} \Big) + O\Big( m\sqrt{n/k} \Big) \, \le \, C t$$
since $k \ge \sqrt{n}(\log n)^2$ and $C$ was chosen sufficiently large. It follows that
$$\Pr(E_S) \, \ge \, \frac{C}{2 \log (n/k)} \, = \, \frac{C}{8 \log (M/m)}$$ 
for each translate $S$ of $S(m)$. Noting that 
$$M \bigg( \log \frac{M}{m} \bigg)^4 \, \le \, t \cdot \frac{\big( \log (n/k) \big)^4}{(n/k)^{1/4}} \, \ll \, t,$$ 
since $k \ll n$, it follows from Lemma~\ref{lem:boot:upper:weak} that, with high probability, $S(M)$ is completely infected by time $3Ct$, as required.
\end{proof}

\section{The lower bound in the pure bootstrap regime}\label{sec:lower:boot}

Having proved the upper bounds in Theorem~\ref{thm:NG2d}, the (significantly harder) task of proving the lower bounds remains. We begin with regime~$(a)$, where we will show that the process is well-approximated by bootstrap percolation. 

\begin{proposition}\label{prop:a:lower}
Let $\delta > 0$, and suppose that $k \ll \log n$. Then
\begin{equation}\label{eq:a:lower}
\tau \geq \left( \frac{\pi^2}{18} - \delta \right) \frac{n}{\log n}
\end{equation}
with high probability as $n \to \infty$. 
\end{proposition}

The proposition generalises (the lower bound of) Holroyd's theorem (which corresponds to the case $k = 1$), and follows by adapting his proof. The bound on $k$ is best possible, since for any $c > 0$ there exists $\delta = \delta(c) > 0$ such that the following holds: if $k = c \log n$ then $\tau < \big( \frac{\pi^2}{18} - \delta \big) \frac{n}{\log n}$ with high probability. (This is simply because the probability of each `double gap' in the growth of a critical droplet, see~\cite[Section~3]{Hol}, decreases by a factor bounded away from 1.)

Rather than give the full (and rather lengthy) details of the proof of Proposition~\ref{prop:a:lower}, we will assume that the reader is familiar with~\cite{Hol} and sketch the key modifications required in our more general setting. We will follow the notation of~\cite{Hol} as much as possible; in particular, we set $\lambda=\pi^2/18$, and choose positive constants $B = B(\delta)$ (large), $Z = Z(B,\delta)$ (small), and $T = T(B,Z,\delta)$ (smaller).  

One of the key notions introduced in~\cite{Hol} is that of a \emph{hierarchy}, which is a way of recording just enough information about the growth of a critical droplet so that a sufficiently strong bound can be given for its probability. Our definition of a hierarchy will be similar to that of Holroyd, but the event that the hierarchy is `satisfied' will have to take 1-neighbour infections into account. The first step is to define the following `random' rectangles process. We will also use this process (though with a different value of $t$) in the next section. 

\begin{defn}[The random rectangles process]\label{def:rrp}
Given $t > 0$, let $A \subset S(n)$ be the set of nucleations that occur in $S(n)$ by time $t$. Run the standard rectangles process until it stops, and then run the following process for time $t$: 
\begin{enumerate}
\item[$(a)$] Wait for a 1-neighbour infection, add this site to the current collection of rectangles, and freeze time. \smallskip
\item[$(b)$] Run the standard rectangles process until it stops.\smallskip 
\item[$(c)$] Unfreeze time, and return to step $(a)$. 
\end{enumerate}
If $R$ is a rectangle, then we will write $I_t(R)$ for the event that $R$ appears in the collection of rectangles at some stage during the random rectangles process.
\end{defn}

The random rectangles process will be a useful tool, allowing us to easily obtain an Aizenman--Lebowitz-type lemma (Lemma~\ref{le:AL:RRP}, below) and to show that good and satisfied hierarchies exist (see below). However, it has the significant disadvantage that $I_t(R)$ is not an increasing event. 
For this reason we will need to define the following event (cf. the random variable $T^*_m(A)$, defined in Section~\ref{sec:AR}), which is implied by $I_t(R)$.

\begin{defn}
Given a rectangle $R$ and $t > 0$, let $I^*_t(R)$ denote the event that $R$ is completely infected by time $t$ in the modified Kesten--Schonmann process, with initially infected set $R \cap A$, in which 2-neighbour infections occur instantaneously. If $I^*_t(R)$ holds then we say that $R$ is \emph{internally filled} by time $t$. 
\end{defn}

It is easy to see that $I_t(R) \Rightarrow I^*_t(R)$. Moreover, by Lemma~\ref{lem:Sn}, with high probability no infection which occurs outside the square $S(n)$ can cause the origin to be infected in time $o(n)$. It will therefore suffice to prove that, with high probability, the origin is not infected in the random rectangles process with $t = ( \lambda - \delta) \frac{n}{\log n}$.

Let us fix $\delta > 0$ and
$$p := \frac{\lambda - \delta}{\log n}$$
for the rest of this section, and note that (if $t = pn$, as above) the elements of $A$ are chosen independently at random with probability $1-e^{-t/n}\leq p$. The main step in the proof of Proposition~\ref{prop:a:lower} 
is the following bound on the probability that a rectangle of `critical' size appears in the random rectangles process. 

\begin{lemma}\label{lem:Bpbound}
Let $R$ be a rectangle with $B/p \le \phi(R) \le 2B/p$. Then 
$$\Pr\big( I_t(R) \big) \, \le \, \exp\bigg( - \frac{2\lambda - \delta}{p} \bigg).$$
\end{lemma}

To prove Lemma~\ref{lem:Bpbound}, we will use the following modification of Holroyd's notion of a `hierarchy'. Given two rectangles $R \subset R'$, let $\Delta^*_t(R,R')$ denote the event that $R'$ is completely infected by time $t$ in the modified Kesten--Schonmann process with initially infected set $R \cup (R' \cap A)$. Given a directed graph $G$ and a vertex $v\in V(G)$, we write $N_G^\to(v)$ for the set of out-neighbours of $v$ in $G$. 


\begin{defn}\label{def:hierarchy}
A \emph{hierarchy} $\H$ for a rectangle $R$ is an ordered pair $\H=(G_\H,D_\H)$, where $G_\H$ is a directed rooted tree such that all of its edges are directed away from the root, and $D_\H \colon V(G_\H) \to 2^{\Z^2}$ is a function that assigns to each vertex of $G_\H$ a rectangle\footnote{We will usually write $D_u$ for the rectangle associated to $u$, instead of the more formal $D_\H(u)$.} such that the following conditions are satisfied:
\begin{enumerate}
\item[$(i)$] The root vertex corresponds to $R$.\smallskip
\item[$(ii)$] Each vertex has out-degree at most 2.\smallskip
\item[$(iii)$] If $v \in N_{G_\H}^\to(u)$ then $D_v \subset D_u$.\smallskip 
\item[$(iv)$] If $N_{G_\H}^\to(u) = \{v,w\}$ then $D_u = \< D_v \cup D_w \>_2$.
\end{enumerate}
We say that $\H$ is \emph{good} if moreover:
\begin{enumerate}
\item[$(v)$] $u \in V(G_\H)$ is a leaf if and only if $\sh(D_u) \leq Z/p$.\smallskip
\item[$(vi)$] If $N_{G_\H}^\to(u)=\{v\}$ and $|N_{G_\H}^\to(v)| = 1$ then $T / p \leq \phi(D_u) - \phi(D_v) \leq 2T / p$.\smallskip
\item[$(vii)$] If $N_{G_\H}^\to(u)=\{v\}$ and $|N_{G_\H}^\to(v)| \ne 1$ then $\phi(D_u) - \phi(D_v) \leq T/p$.\smallskip
\item[$(viii)$] If $N_{G_\H}^\to(u)=\{v,w\}$ then $\phi(D_u)-\phi(D_v) \geq T/p$. 
\end{enumerate}
We say that $\H$ is \emph{satisfied} if the following events all occur disjointly:
\begin{enumerate}
\item[$(a)$] $I^*_t( D_u )$ for every leaf $u \in V(G_\H)$.\smallskip
\item[$(b)$] $\Delta^*_t(D_v,D_u)$ for every pair $\{u,v\}$ such that $N_{G_\H}^\to(u)=\{v\}$.
\end{enumerate}
\end{defn}

The reader should think of the event that $\H$ is satisfied as being `essentially' equivalent to the event that the rectangle $D_u$ appears in the random rectangles process for every vertex $u \in V(G_\H)$. The proof of the following lemma is exactly the same as that of Propositions~31 and~33 of~\cite{Hol}.

\begin{lemma}\label{le:goodsat}
If $R$ is a rectangle that appears in the random rectangles process, then there exists a good and satisfied hierarchy for $R$.
\end{lemma}

Let us write $\H_R$ for the set of all good hierarchies for $R$, and $L(\H)$ for the set of leaves of $G_\H$. We write $\prod_{u \to v}$ for the product over all pairs $\{u,v\} \subset V(G_\H)$ such that $N_{G_\H}^\to(u) = \{v\}$. The following lemma bounds the probability of the event $I_t(R)$.

\begin{lemma}\label{le:boundoverH}
Let $R$ be a rectangle. Then
\begin{equation}\label{eq:boundoverH}
\Pr\big( I_t(R) \big) \leq \sum_{\H \in \H_R} \bigg( \prod_{u \in L(\H)} \Pr\big( I^*_t(D_u)\big) \bigg)\bigg( \prod_{u \to v} \Pr\big( \Delta^*_t(D_v,D_u) \big) \bigg).
\end{equation}
\end{lemma}

\begin{proof}[Proof of Lemma~\ref{le:boundoverH}]
Since the events $I^*_t(D_u)$ for $u \in L(\H)$ and $\Delta^*_t(D_v,D_u)$ for $u \to v$ are increasing and occur disjointly, this follows from Lemma~\ref{le:goodsat} and the van den Berg--Kesten inequality.
\end{proof}

\pagebreak

The probability that a seed is internally filled is easily bounded by the following lemma.

\begin{lemma}\label{lem:seeds}
If $\phi(R) \le Z/p$ then $\Pr\big( I^*_t(R) \big) \le e^{-B\phi(R)}$. 
\end{lemma}

\begin{proof}
Note first that $R$ contains at least $\phi(R) / 4$ nucleations with probability at most
$${\phi(R)^2 \choose \phi(R)/4} \bigg( \frac{t}{n} \bigg)^{\phi(R)/4} \, \le \, \big( 4e p \cdot \phi(R)\big)^{\phi(R)/4}  \, \le \, (4eZ)^{\phi(R)/4} \, \le \, e^{-B\phi(R)}$$
where $t = pn$ and since $Z = Z(B)$ was chosen sufficiently small. On the other hand, the probability that $I^*_t(R)$ occurs and $R$ contains at most $\ell = \phi(R) / 4$ nucleations is at most
$$\Pr\bigg( \sum_{i = 2\ell}^{\phi(R)} Y_i \le t \bigg) \, \le \, e^{-B\phi(R)}$$
where $Y_i \sim \Exp(n/2ki)$ are independent. Note that the final inequality holds since $n/k \gg t$, and thus $\sum_{i = 2\ell}^{\phi(R)} Y_i \le t$ implies that $Y_i = o\big(\Ex[Y_i]\big)$ for at least half of the variables $Y_i$ (cf.~the proof of Lemma~\ref{lem:Sn}). 
\end{proof}

Bounding $\Pr\big( \Delta^*_t(D_v,D_u) \big)$ when $N_{G_\H}^\to(u) = \{v\}$ is not so straightforward. Before stating the bound we will prove, we need a couple of definitions. 

\begin{defn}
Let $R \subset R'$ be rectangles with $R = [a_1,a_2] \times [b_1,b_2]$ and $R' = [a'_1,a'_2] \times [b'_1,b'_2]$. A \emph{double gap} in the annulus $R' \setminus R$ is a set $X$ such that 
\begin{itemize}
\item[$(a)$] $X = [a_1,a_2] \times \{m,m+1\}$ or $X = \{m,m+1\} \times [b_1,b_2]$ for some $m \in \Z$;
\item[$(b)$] $X$ intersects $R'$;
\item[$(c)$] $X$ is disjoint from $A \cup R$. 
\end{itemize}
Let $Q_p(R,R')$ denote the probability that there is no double gap in $R' \setminus R$.
\end{defn}

Note that if there is no double gap in $R' \setminus R$, then the event $\Delta^*_t(R,R')$ occurs. We will require the following bound on the probability of the event $\Delta^*_t(R,R')$. 

\begin{lemma}\label{lem:approx:double:gaps}
Let $R \subset R'$ be rectangles, and suppose that $Z/p \le \phi(R) \le 2B/p$ and $\phi(R') - \phi(R) \leq T/p$. Then
\begin{equation}\label{eq:approx:double:gaps}
\Pr\big( \Delta^*_t(R,R') \big) \le n^{o(1)} Q_p(R,R')^{1-\delta^2}.
\end{equation}
\end{lemma}

\begin{proof}
The proof is similar to that of~\cite[Proposition~22]{Hol}, the key idea being to partition the space according to the rows and columns of $R' \setminus R$ that contain either a nucleation in the corner regions (as in~\cite{Hol}), or a site that is infected via a 1-neighbour infection. We will prove that (with sufficiently high probability) only a small proportion of the rows and columns satisfy either condition;\footnote{More precisely, a small proportion of the rows and columns satisfy the first condition, and $o(\log n)$ of them satisfy the second.} the lemma will then follow via a short (and standard) calculation.  

To be precise, recall that $p = \frac{\lambda - \delta}{\log n}$, and let $\eps = \eps(n) > 0$ satisfy
\begin{equation}\label{eq:eps}
\eps \ll T \qquad \text{and} \qquad e^{-10/\eps} \gg \frac{k}{\log n},
\end{equation}
(so, in particular, $\eps(n) \to 0$ as $n \to \infty$), and let $E$ denote the event that at least $\eps \log n$ sites of $R' \setminus R$ are infected via 1-neighbour infections. (Recall that our initial infected set is $R \cup (R' \cap A)$, and we are running the modified Kesten--Schonmann process, i.e., allowing 1-neighbour infections and performing 2-neighbour infections instantaneously, for time $t$.) To bound the probability of this event, observe that
$$\Pr(E) \le \, \sum_{\ell = 0}^{\infty} \Pr\Big( \big| (R' \setminus R) \cap A \big| = \ell \Big) \Pr\bigg( \sum_{i = \phi(R) + 2\ell}^{\phi(R) + 2\ell + \eps \log n} Y_i \le t \bigg),$$
where the $Y_i \sim \Exp(n/2ki)$ are independent random variables, as usual. Now, since the area of $R' \setminus R$ is at most $(\log n)^2$, by our choice of $T$, we have 
$$\Pr\big( |(R' \setminus R) \cap A| = \ell \big) \, \le \, {(\log n)^2 \choose \ell} p^\ell \, \le \, \bigg( \frac{e(\log n)^2 p}{\ell} \bigg)^\ell \, \le \, e^{-\ell}$$
for every $\ell \ge 10\log n$, and $\ds\sum_{\ell = 10\log n}^\infty e^{-\ell} \le \frac{1}{n^3}$. On the other hand, if $\ell < 10\log n$ then 
$$\Pr\bigg( \sum_{i = \phi(R) + 2\ell}^{\phi(R) + 2\ell + \eps \log n} Y_i \le t \bigg) \, \le \, \Pr\bigg( \sum_{i = 5B\log n}^{(5B+ \eps) \log n} Y_i \le t \bigg).$$ 
To bound this expression, note that the second condition in~\eqref{eq:eps} implies that $n/k \gg te^{10/\eps}$, so the sum may be coupled from below by a sum of $\eps \log n$ independent $\Exp(t e^{10/\eps} / \log n)$ random variables. Thus,
\[
\Pr\bigg( \sum_{i = 5B\log n}^{(5B+ \eps) \log n} Y_i \le t \bigg) \, \le \, \bigg(\frac{e}{\eps \cdot e^{10/\eps}}\bigg)^{\eps \log n} \, \le \, \frac{1}{n^4},
\]
by Lemma~\ref{lem:poisson}.
Summing over the choices of $\ell < 10\log n$ and combining with the calculation for $\ell \geq 10\log n$, it follows that 
\begin{equation}\label{eq:marameo1}
\Pr(E) \le \frac{2}{n^3} \ll \bigg( \frac{Z}{2} \bigg)^{T/p} \le \bigg( \frac{Z}{2} \bigg)^{\phi(R') - \phi(R)} \le Q_p(R,R'),
\end{equation}
since $T = T(Z)$ is sufficiently small, $p = \frac{\lambda - \delta}{\log n}$, each row (or column) of length at least $Z/p$ contains an element of $A$ with probability at least $Z/2$, and recalling that $Q_p(R,R')$ is the probability that there is no double gap in $R' \setminus R$. 

We will now take a union bound over pairs $(L_1,L_2)$, where $L_1$ and $L_2$ denote (respectively) the set of rows and columns of $R' \setminus R$ that contain a 1-neighbour infection. 
Hence, writing $\LL$ for the collection of pairs $(L_1,L_2)$ of sets of rows and columns of $R' \setminus R$ such that $|L_1| + |L_2| \le 2\eps \log n$, we have 
\begin{equation}\label{eq:marameo2}
\Pr\big( \Delta^*_t(R,R') \big) \, \le \, \Pr(E) + \sum_{(S_1,S_2) \in \LL} \Pr\big( \Delta_t^*(R,R') \cap \{ L_1 = S_1 \} \cap \{L_2 = S_2 \} \big).
\end{equation}
Now, to bound the terms in the sum on the right-hand side, we simply repeat the calculation from~\cite{Hol}. To be more precise, we partition according to the pair $(M_1,M_2)$, where $M_1$ and $M_2$ denote (respectively) the set of rows and columns of $R' \setminus R$ whose intersection with the corner regions of $R' \setminus R$ contains a nucleation. We claim that, for each pair $(S_1,S_2) \in \LL$, we have
\begin{equation}\label{eq:marameo3}
\Pr\big( \Delta_t^*(R,R') \cap \{ L_1 = S_1 \} \cap \{L_2 = S_2 \} \big) \le n^{o(1)} \sum_{(M_1,M_2)} \bigg( \frac{2\sqrt{T}}{Z} \bigg)^{|M_1| + |M_2|} Q_p(R,R').
\end{equation}
Indeed, observe first that if $\Delta_t^*(R,R')$ holds, then the rows not in $S_1 \cup M_1$ and columns not in $S_2 \cup M_2$ \emph{cannot} contain a double gap. This event has probability at most 
$$\bigg( \frac{2}{Z} \bigg)^{|S_1| + |M_1| + |S_2| + |M_2|} Q_p(R,R') = n^{o(1)} \bigg( \frac{2}{Z} \bigg)^{|M_1| + |M_2|} Q_p(R,R'),$$
since $|S_1| + |S_2| = o(\log n)$, each row (or column) of length at least $Z/p$ contains an element of $A$ with probability at least $Z/2$, and recalling that $\phi(R) \ge Z/p$. Moreover, since $\phi(R') - \phi(R) \leq T/p$, the probability that the rows and columns of $(M_1,M_2)$ contain a nucleation is at most $T^{\max\{ |M_1|, |M_2| \}}$, and hence we obtain~\eqref{eq:marameo3}. 

Finally, observe that $Q_p(R,R') \le \big( 1 - e^{-3B} \big)^{(\phi(R') - \phi(R))/2}$, since $\phi(R) \le 2B/p$. Since $T = T(B,Z,\delta)$ is sufficiently small, it follows that 
$$\sum_{(M_1,M_2)} \bigg( \frac{2\sqrt{T}}{Z} \bigg)^{|M_1| + |M_2|} \le \big( 1 - e^{-3B} \big)^{-\delta^2(\phi(R') - \phi(R))/2} \le Q_p(R,R')^{-\delta^2}.$$
Hence, combining~\eqref{eq:marameo1},~\eqref{eq:marameo2} and~\eqref{eq:marameo3}, and recalling that $|\LL| = n^{o(1)}$, we obtain
$$\Pr\big( \Delta^*_t(R,R') \big) \, \le \, Q_p(R,R') + n^{o(1)} \sum_{(S_1,S_2) \in \LL} Q_p(R,R')^{1-\delta^2} \le n^{o(1)} \cdot Q_p(R,R')^{1-\delta^2},$$
as required.
\end{proof}

The proof of Lemma~\ref{lem:Bpbound} now follows from the calculation in~\cite{Hol}.

\begin{proof}[Sketch proof of Lemma~\ref{lem:Bpbound}]
Combining Lemmas~\ref{le:boundoverH},~\ref{lem:seeds} and~\ref{lem:approx:double:gaps}, and noting that the number of vertices in a hierarchy is bounded, it follows that
\begin{equation}\label{eq:Bpbound:proof}
\Pr\big( I_t(R) \big) \le n^{o(1)} \sum_{\H \in \H_R} \bigg( \prod_{u \in L(\H)} e^{-B\phi(R)} \bigg)\bigg( \prod_{u \to v} Q_p(R,R') \bigg)^{1-\delta^2}.
\end{equation}
Holroyd~\cite{Hol} proved that (roughly speaking) the sum on the right-hand side of~\eqref{eq:Bpbound:proof} is dominated by hierarchies which correspond to a single droplet that is always approximately square. More precisely (and more importantly for us), he proved (see~\cite[Section~10]{Hol}) that it is at most $\exp\big( - (2\lambda - \delta) / p \big)$, as required.
\end{proof}


Finally, let us note the following Aizenman--Lebowitz-type lemma, which follows immediately from the definition.

\begin{lemma}\label{le:AL:RRP}
If $R$ appears in the random rectangles process, then the following holds for every $1 \leq \ell \leq \phi(R)$. There exists a rectangle $R' \subset R$ which appears in the random rectangles process with $\ell \leq \phi(R') \leq 2\ell$.
\end{lemma}

We can now complete the proof of Proposition~\ref{prop:a:lower}. 

\begin{proof}[Proof of Proposition~\ref{prop:a:lower}]
Let $R$ be the rectangle that infects the origin in the random rectangles process. If $\phi(R) \ge 2B/p$, then, by Lemma~\ref{le:AL:RRP}, there exists a rectangle $R'$ such that $I_t^*(R')$ holds and $B/p \le \phi(R') \le 2B/p$. By Lemma~\ref{lem:Bpbound}, and recalling that $p = \frac{\lambda - \delta}{\log n}$, it follows that  
$$\Pr\big( I_t(R') \big) \, \le \, \exp\bigg( - \frac{2\lambda - \delta}{p} \bigg) \, \le \, \frac{1}{n^{2+\delta}},$$
so with high probability there does not exist such a rectangle, and we are done. 

So assume that $\phi(R) = m \le 2B/p$. 
If $m \le \log\log n$ then there must be a nucleation within this distance of the origin by time $t$, and with high probability this does not occur. Next, if $\log\log n \le m \le Z/p$, then
\[
\Pr\big( I_t(R) \big) \le \Pr\big( I^*_t(R) \big) \le e^{-Bm} \ll p^2
\]
by Lemma~\ref{lem:seeds}, so we are done in this case also by taking the union bound over the $O(p^{-2})$ choices of $R$. Finally, if $Z/p \le m \le 2B/p$, then by Lemma~\ref{le:AL:RRP} there exists a rectangle $R'$ such that $I_t(R')$ holds and $Z/p \le \phi(R) \le 2Z/p$, in which case we are again done by Lemma~\ref{lem:seeds}. This completes the proof of Proposition~\ref{prop:a:lower}.
\end{proof}

\section{The lower bound in the accelerating regime}\label{sec:lowerAC}

In the section we will prove the following proposition, which provides the lower bound of Theorem~\ref{thm:NG2d} in regime~$(b)$; that is, the regime in which the acceleration phase of growth dominates. 

\begin{proposition}\label{prop:b:lower}
Let $\delta > 0$, and suppose that $k \gg \log n$. Then
\begin{equation}\label{eq:b:lower}
\tau \geq \left(\frac{1}{4} - \delta\right) \frac{n}{k} \log \bigg( \frac{k}{\log n} \bigg)
\end{equation}
with high probability.
\end{proposition}

The main tools in the proof of Proposition~\ref{prop:b:lower} will be Lemma~\ref{lem:acc:lower}, which gives an upper bound on the typical growth of a droplet in the (modified) Kesten--Schonmann process  and the random rectangles process (see Definition~\ref{def:rrp}), applied with $t = \big(\frac{1}{4} - \delta \big) \frac{n}{k} \log \big( \frac{k}{\log n} \big)$. Before giving the details, let us sketch the basic strategy. Let $R \subset S(n)$ be the first rectangle that occurs in the random rectangles process and contains the origin. We will divide into three cases depending on whether the semi-perimeter $m = \phi(R)$ is bigger than $\sqrt{k\log n}$, much smaller than $\sqrt{n/t}$, or somewhere in between. 

The easiest case is when $m \ll \sqrt{n/t}$, since with high probability there are no nucleations within distance $m$ of the origin by time $t$. The other two cases are harder, but similar to one another. First, if $m \ge \sqrt{k \log n}$ then we will apply Lemma~\ref{le:AL:RRP} to find a rectangle $R'$ with semi-perimeter $m \approx \sqrt{k\log n}$ that appears in the random rectangles process. If $R'$ contains more than $m^2 (\log k) / k$ nucleations, then we will bound the probability using Lemma~\ref{lem:countingnucleations}; if not then we will apply Lemma~\ref{lem:acc:lower}. In either case, the bounds obtained will be super-polynomial in $n$, so we can bound the probability that such a rectangle $R$ exists using Markov's inequality. 

On the other hand, if 
$$\sqrt{ \frac{n}{t} } \, \approx \, \sqrt{ \frac{k}{\log \big( \frac{k}{\log n} \big) } } \, \lesssim \, m \, \le \, \sqrt{k\log n}$$
then we will apply Lemma~\ref{le:AL:RRP} to find a $t$-spanned rectangle $R'$ \emph{within distance $\sqrt{k\log n}$ of the origin} with semi-perimeter roughly $\sqrt{k / \log \big( \frac{k}{\log n} \big)}$. If $R'$ contains more than $\log k$ nucleations, then we will bound the probability using Lemma~\ref{lem:countingnucleations}; if not then we will apply Lemma~\ref{lem:acc:lower}. Crucially, the number of such rectangles is only polynomial in $k$, so our bounds on the probability, which are polynomially small in $k$, are sufficient for an application of Markov's inequality.

\begin{proof}[Proof of Proposition~\ref{prop:b:lower}]
Let $R$ be the first rectangle that occurs in the random rectangles process (with $t = \big(\frac{1}{4} - \delta \big) \frac{n}{k} \log \big( \frac{k}{\log n} \big)$) and contains the origin, as in the sketch above, and let $\eps = \eps(n)$ be a function that tends to zero sufficiently slowly. Since with high probability there are no nucleations within distance $\eps \sqrt{n/t}$ of the origin by time $t$, we may assume that 
$$\phi(R) \,\ge\, \eps \sqrt{ \frac{n}{t} } \,\ge\, 2\eps \sqrt{\frac{k}{\log \big( \frac{k}{\log n} \big)}}.$$

Suppose first that $\phi(R) \ge 2\sqrt{k \log n}$. In this case we may apply Lemma~\ref{le:AL:RRP} to find a rectangle $R' \subset S(n)$ with semi-perimeter 
$$\sqrt{k \log n} \, \le \, m \, \le \, 2\sqrt{k \log n}$$
that appears in the random rectangles process. Now, setting $\ell = (m^2/k) \log \big( \frac{k}{\log n} \big)$, we have
$$P_t(R',j) \le n^{-\log (\frac{k}{\log n})}$$
for any $j\ge\ell$ by Lemma~\ref{lem:countingnucleations}, and so (since $k \gg \log n$ and the number of choices of $j\ge\ell$ and the number of rectangles in $S(n)$ are only polynomial in $n$) we may assume that $R'$ contains at most $\ell$ nucleations by time $t$. As in Definition~\ref{def:rrp}, let $A$ be the set of nucleations that occur in $R$ by time $t$, so $|A|\le\ell$. Note that
$$\ell \, = \, \frac{m^2 \log \big( \frac{k}{\log n} \big)}{k} \, \le \, 4\log\bigg( \frac{k}{\log n} \bigg) \log n \, \ll \, m,$$ 
where in the final step we used our assumption that $k \gg \log n$, and it follows by Lemma~\ref{lem:acc:lower} that 
\begin{equation}\label{eq:lem:acc:lower:app}
\P\bigg( T^*_m(R' \cap A) \leq (1-\delta)\frac{n}{2k}\log\frac{m}{\ell} \bigg) \, \leq \, \exp\Big(- \delta^2 \max\big\{ m^{\delta/2}, \ell \big\} \Big).
\end{equation}
Now, since $m/\ell = \big( k / \log n \big)^{1/2 + o(1)}$, it follows that
$$(1-\delta)\frac{n}{2k}\log\frac{m}{\ell} \, \ge \, (1-2\delta)\frac{n}{4k}\log \bigg( \frac{k}{\log n} \bigg) \, > \, t,$$ 
and hence the probability that $R'$ appears in the random rectangles process by time $t$, given that $A$ is the set of nucleations in $R'$ and $|A|\le\ell$, is at most
\[
\exp\Big(- \delta^2 \max\big\{ m^{\delta/2}, \ell \big\} \Big) \, \le \, n^{- \delta^2 \log (\frac{k}{\log n})},
\]
where the inequality follows since
\[
\max\big\{ m^{\delta/2}, \ell \big\} \, \ge \, \ell \, \ge \, \log n \cdot \log\left(\frac{k}{\log n}\right).
\]
This probability bound is again super-polynomial in $n$, and it is uniform in $A$, so by Markov's inequality this completes the proof in the case $\phi(R) \ge 2\sqrt{k \log n}$.

Finally, suppose that $\eps \sqrt{n/t} \le \phi(R) \le 2\sqrt{k \log n}$. By Lemma~\ref{le:AL:RRP} there exists a rectangle $R' \subset R$ within distance $2\sqrt{k \log n}$ of the origin, with semi-perimeter 
$$\eps \sqrt{\frac{k}{\log \big( \frac{k}{\log n} \big)}} \, \le \, m \, \le \, 2\eps \sqrt{\frac{k}{\log \big( \frac{k}{\log n} \big)}},$$
that appears in the random rectangles process. Setting $\ell = \log k$, we have
$$P_t(R',j) \le k^{-\log\log k}$$
for any $j \ge \ell$, by Lemma~\ref{lem:countingnucleations}. Thus, since the number of rectangles of semi-perimeter at most $k$ within distance $2\sqrt{k \log n}$ of the origin is only polynomial in $k$, and so is the number of choices of $j\ge\ell$, we may assume that $R'$ contains at most $\ell$ nucleations by time $t$. But $\ell = \log k \ll m$ and hence, by Lemma~\ref{lem:acc:lower}, we once again have~\eqref{eq:lem:acc:lower:app}. Since again $m/\ell \ge \big( k / \log n \big)^{1/2 + o(1)}$, it follows as before that the probability that $R'$ appears in the random rectangles process by time $t$, given that $A$ is the set of nucleations in $R'$ and $|A|\le\ell$, is at most 
$$\exp\Big(- \delta^2 \max\big\{ m^{\delta/2}, \ell \big\} \Big) \, = \, \exp\big( - k^{\Omega(\delta)} \big),$$
which is super-polynomial in $k$ (and uniform in $A$). This completes the proof of the proposition.
\end{proof}

\section{The lower bound at terminal velocity}\label{sec:lowerTV}

In this section we will complete the proof of Theorem~\ref{thm:NG2d} by proving the following lower bound in regime~$(c)$.

\begin{proposition}\label{prop:c:lower}
There exists $c > 0$ such that if $\sqrt{n}(\log n)^2 \ll k \ll n$, then
\begin{equation}\label{eq:prop:c:lower}
\tau \geq c \cdot \left(\frac{n^2}{k\log (n/k)}\right)^{1/3}
\end{equation}
with high probability.
\end{proposition}

The broad structure of the proof here is the same as in the previous section, but because the time for two-neighbour infections is important in this (the terminal velocity) regime, we cannot take instantaneous bootstrap closures. As a result, the coupling we will use to obtain an Aizenman--Lebowitz-type lemma is somewhat more complicated. Fortunately, since we are only aiming to determine $\tau$ up to a constant factor, we can afford to be rather generous. 

Let us fix $t = c \cdot \big(\frac{n^2}{k\log(n/k)} \big)^{1/3}$ throughout this section, and also a function $r = r(n)$ such that $r/t \to \infty$ sufficiently slowly. Note that, by Lemma~\ref{lem:Sn}, we may assume that no infections occur outside the square $S(r)$. 

\begin{defn}[The generous rectangles process]
Let $A$ be the set of nucleations that occur in $S(r)$ by time $t$, and define an initial collection of rectangles 
$$\big\{ x + S\big( 100 n^{1/4} \big) : x \in A \big\}$$ 
by placing a copy of $S\big(100 n^{1/4} \big)$ on each nucleation. 

Now repeat the following steps $\lceil tk^{1/2}/n^{3/4} \rceil$ times:
\begin{enumerate}
\item If there exist two rectangles that are within $\ell_1$-distance $500n^{1/4}$ of each other, then choose such a pair and replace them by the smallest rectangle containing both. Continue iterating this step until all pairs of rectangles are $\ell_1$-distance at least $500n^{1/4}$ apart.
\item Enlarge each side of each rectangle\footnote{As before, the centre of each rectangle remains the same, and the length of each side of each rectangle increases by $200 n^{1/4}$.} by distance $200n^{1/4}$.
\end{enumerate}
We say that a rectangle is \emph{generously spanned} if it appeared at some point in this process. 
\end{defn}

The following Aizenman--Lebowitz-type lemma follows immediately from the definition above. 

\begin{lemma}\label{le:AL:TV}
If $R$ is a generously spanned rectangle, then for every $1\leq \ell \leq \phi(R)$, there exists a generously spanned rectangle $R' \subset R$ such that 
$$\ell \, \le \, \phi(R') \, \le \, 2\ell + 500n^{1/4}.$$
\end{lemma}

The following key lemma was essentially proved in Section~\ref{sec:TV}.

\begin{lemma}\label{lem:TV:coupled}
Let $B$ denote the set of sites infected by time $t$ in the nucleation and growth process run on $S(r)$. Then, with high probability, $B$ is contained in the union of the final set of rectangles in the generous rectangles process.
\end{lemma}

\begin{proof}
Consider the minimum $1 \le i \le \lceil tk^{1/2}/n^{3/4} \rceil$ for which the following holds: there exists a site that is infected by time $i \cdot n^{3/4} / k^{1/2}$ in the nucleation and growth process, but is not contained in the union of the set of rectangles after $i$ steps of the generous rectangles process. This implies that one of the rectangles at the previous step grew in some direction by more than $100 n^{1/4}$ within time $n^{3/4} / k^{1/2}$, and by Lemma~\ref{lem:tv:lower} this event has probability at most $n^{-3}$ for a fixed rectangle. Since the number of choices for $i$ and the rectangle is at most $r^4 \cdot t k^{1/2}/n^{3/4} \ll n^3$ (by our choices of $t$, $k$ and $r$), the lemma follows by Markov's inequality. 
\end{proof}

We are now ready to prove the proposition. 

\begin{proof}[Proof of Proposition~\ref{prop:c:lower}]  
If $\tau \le t$ then, by Lemma~\ref{lem:TV:coupled}, we may assume that there exists a generously spanned rectangle $R$ containing the origin. Let $\eps = \eps(n)$ be a function which tends to zero sufficiently slowly, and observe that, with high probability, there are no nucleations within distance $\eps \sqrt{n/t}$ of the origin by time $t$. We may therefore moreover assume that 
$$\phi(R) \,\ge\, \eps \sqrt{ \frac{n}{t} } \, \ge \, \eps \big( k n \log (n/k) \big)^{1/6}.$$

Suppose first that $\phi(R) \ge M = (kn)^{1/6} \big( \log (n/k) \big)^{2/3}$. In this case we may apply Lemma~\ref{le:AL:TV} to find a generously spanned rectangle $R' \subset S(r)$ with semi-perimeter 
$$M/8 \, \le \, \phi(R') \, \le \, M/4  + 500n^{1/4} \, \le \, M/2.$$
We claim that $R'$ contains at least $\log (n/k)$ nucleations. To see this, observe that, since $k \gg \sqrt{n} (\log n)^2$ and $c$ is sufficiently small, each nucleation contributes at most 
\begin{equation}\label{eq:nuccontribution}
O\big( n^{1/4} + t \sqrt{ k/n } \big) \, \le \, \frac{(kn)^{1/6}}{8 (\log(n/k))^{1/3}} \, = \, \frac{M}{8\log(n/k)}
\end{equation}
to the semi-perimeter of a generously spanned rectangle. Thus, if $R'$ is generously spanned, then for some $(i,j) \in \Z^2$ with $-r/M \le i, j \le r/M$, the square $S(M) + (iM,jM)$ must contain at least $\log (n/k)$ nucleations. But $r/M \le \sqrt{n/k}$, so by Lemma~\ref{lem:TVnucleations} and the union bound, this event has probability $o(1)$, as required.

Finally suppose that 
$$\eps \big( k n \log (n/k) \big)^{1/6} \, \le \, \phi(R) \, \le \, M.$$
In this case we may apply Lemma~\ref{le:AL:TV} to find a generously spanned rectangle $R' \subset S(M)$ with semi-perimeter
\[
\eps m \le \phi(R') \le 2 \eps m + 500n^{1/4} \le 3\eps m,
\]
where $m = \big( k n \log (n/k) \big)^{1/6} \gg n^{1/4}$. Now recall from~\eqref{eq:nuccontribution} that each nucleation contributes at most
$$\frac{(kn)^{1/6}}{(\log(n/k))^{1/3}} \, = \, \frac{m}{(\log(n/k))^{1/2}} \, \ll \, \frac{\eps m}{(\log(n/k))^{1/3}}$$
to the semi-perimeter of a generously spanned rectangle, and hence $R'$ contains at least $\big( \log (n/k) \big)^{1/3}$ nucleations. Thus, if $R'$ is generously spanned, then for some $(i,j) \in \Z^2$ with $-M/m \le i, j \le M/m$, the square $S(m) + (im,jm)$ must contain at least $\big( \log (n/k) \big)^{1/3}$ nucleations. But $M/m = \big( \log (n/k) \big)^{1/2}$, so by Lemma~\ref{lem:TVnucleations} and the union bound, this event also has probability $o(1)$. Thus $\tau \ge t$ with high probability, which completes the proof of the proposition.
\end{proof} 

We are finally ready to put together the pieces and prove Theorem~\ref{thm:NG2d}. 

\begin{proof}[Proof of Theorem~\ref{thm:NG2d}]
In regime~$(a)$, the upper bound was proved in Proposition~\ref{prop:a:upper}, and the lower bound in Proposition~\ref{prop:a:lower}. In regime~$(b)$, the upper bound was proved in Proposition~\ref{prop:b:upper}, and the lower bound in Proposition~\ref{prop:b:lower}. Finally, in regime~$(c)$, the upper bound was proved in Proposition~\ref{prop:c:upper}, and the lower bound in Proposition~\ref{prop:c:lower}. 
\end{proof}

\section{Higher dimensions and more general update rules}\label{sec:problems}

\subsection{Higher dimensions}\label{sec:higher:dims}

In this paper we have shown how relatively simple techniques from bootstrap percolation can be used to prove surprisingly sharp bounds on the relaxation time of the nucleation and growth process in two dimensions. In higher dimensions (with more than two different rates) the problem is substantially more difficult, and a result corresponding to that of Dehghanpour and Schonmann~\cite{DS97a} was obtained only recently by Cerf and Manzo~\cite{CM13a}. They consider the nucleation and growth process on $\Z^d$ where vertices with $i$ already-infected neighbours are infected at rate $k_i/n$, where the functions $k_i = k_i(n)$ satisfy $1 = k_0 \le k_1 \le \cdots \le k_d = n$. The main result of~\cite{CM13a} states that, with high probability, 
$$\tau = n^{1 + o(1)} \bigg( \min\Big\{ k_1^{d+1}, \, k_2^d, \, k_1 k_3^{d-1}, \, k_1 k_2 k_4^{d-2}, \ldots, \, k_1 \cdots k_{d-2} k_{d}^2, \, k_1 \cdots k_d \Big\} \bigg)^{-1/(d+1)}.$$
We remark that the proof, which is highly non-trivial, also used techniques from bootstrap percolation (in particular, those introduced in~\cite{CC,CM}).  The techniques introduced in~\cite{CM13a} were extended by the same authors in~\cite{CM13b} to determine the relaxation time of the Ising model on $\Z^d$ at low temperatures. 

It is natural to ask whether or not the above bounds on $\tau$ can be strengthened along the lines of Theorem~\ref{thm:NG2d}. We expect, however, that the following problem will be significantly harder to resolve in general than it was in two dimensions. 

\begin{prob}\label{prob:ddim}
Determine the relaxation time of the $d$-dimensional nucleation and growth process up to a constant factor.
\end{prob}

It is possible that one could even prove a sharp threshold for the relaxation time for some ranges of $\k = (k_1,\ldots,k_{d-1})$, as in Theorem~\ref{thm:NG2d}. It would also be very interesting to do so in two dimensions in regime~$(c)$.

\begin{prob}
Determine the sharp threshold (if one exists) for the relaxation time in two dimensions when $\sqrt{n}(\log n)^{2} \ll k \ll n$. 
\end{prob}

\subsection{Monotone cellular automata}

The following very significant generalization of bootstrap percolation on $\Z_n^d$ was recently introduced by Bollob\'as, Smith and Uzzell~\cite{BSU}: given an arbitrary finite collection $\U = \{ X_1,\ldots,X_m \}$ of finite subsets of $\Z^d\setminus\{\0\}$, and a set $A \subset \Z_n^d$ of initially infected sites, set $A^{(0)} = A$ and
$$A^{(i+1)} \, = \, A^{(i)} \cup \big\{ v \in \Z_n^d \,:\, v + X \subset A^{(i)} \text{ for some } X \in \U \big\}$$
for each $i \ge 0$. Thus, a site $v$ becomes infected at step $i+1$ if the translate by $v$ of one of the sets of the update family $\U$ is already entirely infected at step $i$, and infected sites remain infected forever. It was shown in~\cite{BBPS,BSU} that in two dimensions one of the following three possibilities holds:\footnote{Here $p_c(\Z_n^2,\U)$ is defined to be the infimum over $p$ such that a $p$-random set percolates under the $\U$-bootstrap process on $\Z_n^2$ with probability at least $1/2$, cf.~\eqref{def:pc:r:boot}.}
\begin{itemize}
\item $\U$ is `supercritical' and $p_c(\Z_n^2,\U) = n^{-\Theta(1)}$.\medskip
\item $\U$ is `critical' and $p_c(\Z_n^2,\U) = (\log n)^{-\Theta(1)}$; \medskip 
\item $\U$ is `subcritical' and $\displaystyle \liminf_{n\to\infty} \, p_c(\Z_n^2,\U) > 0$;  
\end{itemize}
Furthermore, in~\cite{BDMS} the value of $p_c(\Z_n^2,\U)$ was determined up to a constant factor for all critical update families. 

Motivated by these results, it is natural to define the following nucleation and growth process for arbitrary update rules.

\begin{defn}[A random monotone cellular automaton]
Given an arbitrary family $\U_1,\ldots,\U_r$ of update rules in $\Z^d$ (i.e., each $\U_j$ is a finite collection of finite subsets of $\Z^d\setminus\{\0\}$), and functions $1 \le k_1(n) \le \ldots \le k_r(n) = n$, let a site $v \in \Z^d$ be infected at rate $k_\ell(n)/n$ at time $t$ if 
$$\ell \, = \, \max\big\{ j \in [r] : v + X \subset A_t \text{ for some } X \in \U_j \big\},$$ 
where $A_t$ is the set of infected sites at time $t$, and at rate $1/n$ otherwise. 
\end{defn}

We expect the following problem to already be extremely difficult.

\begin{prob}
Determine the relaxation time of an arbitrary random monotone cellular automaton up to a constant factor when $d = r = 2$. 
\end{prob}

As a first step, it would be interesting to understand the rough behaviour of the relaxation time when $\U_1$ is a supercritical update family and $\U_2$ is critical.


\end{document}